\documentclass[11pt,a4paper]{article}
 \usepackage{amsthm}
 \usepackage{chngcntr} 

 \usepackage{amsfonts}
 
 \usepackage{graphics}
 \usepackage{indentfirst}
 \usepackage{latexsym}
 \usepackage{amsmath}
 \usepackage{amssymb}
 \usepackage[dvips]{epsfig}
 \usepackage{amscd}
 \usepackage{amsthm}
 \usepackage{hyperref}
 \hypersetup{
 	colorlinks=true,
 	linkcolor=blue,
 	anchorcolor=blue,
 	citecolor=blue}
 
 \usepackage{cite}
 
 \newtheorem{theorem}{Theorem}[section]
 \newtheorem{remark}{Remark}[section]
 
 \newtheorem{lemma}[theorem]{Lemma}
 \newtheorem{pro}{Proposition}[section]
 \newtheorem{cor}{Corollary}[section]

 \newcommand{\bl}{\begin{lemma}}
 	\newcommand{\el}{\end{lemma}}
 \newcommand{\et}{\end{theorem}}

\newcommand{\la}{\label}

\newcommand{\bn}{\begin{eqnarray}}
\newcommand{\en}{\end{eqnarray}}
\newcommand{\bnn}{\begin{eqnarray*}}
\newcommand{\enn}{\end{eqnarray*}}

\newcommand{\bnnn}{\begin{eqnarray*}}
\newcommand{\ennn}{\end{eqnarray*}}

\newcommand{\ba}{\begin{aligned}}
\newcommand{\ea}{\end{aligned}}
\newcommand{\be}{\begin{equation}}
\newcommand{\ee}{\end{equation}}

\def\norm[#1]#2{\|#2\|_{#1}}

\def\la{\label}

\makeatletter      
\@addtoreset{equation}{section}
\makeatother       
\title{The Outer Pressure Problem for the Compressible Navier--Stokes System
with Temperature-Dependent Transport Coefficients}
\date{ }

\allowdisplaybreaks[4]
\begin{document}
\author{Hongyu Wang$^a$, Rong Zhang$^b$ \thanks{
		Email addresses: {hyuwang\_a@163.com} (H. Y. Wang),  rzhang0921@gmail.com (R.
		Zhang).} \\[3mm]  a. School of Mathematics and Computer Science,\\ Nanchang University,  Nanchang 330031, P. R. China; \\
	b. School of Mathematics and Computer Science, \\Nanchang University,  Nanchang 330031, P. R. China.}
\maketitle
\begin{abstract}
We study global strong solutions and their large-time behavior for the
one-dimensional outer pressure problem of a viscous, heat-conducting
ideal polytropic gas. The viscosity and heat conductivity satisfy
$\mu=\tilde\mu\theta^\alpha$ and $\kappa=\tilde\kappa\theta^\beta$.
For each prescribed nonnegative pressure and each fixed $\beta\geq0$,
we allow large $H^1$ initial data with positive specific volume and
temperature, provided that $\alpha\geq0$ is sufficiently small.
When the limiting outer pressure is positive, the solution has uniform
bounds and converges to a stationary state. When the limiting outer pressure
is zero and the stated weighted pressure conditions hold, the
normalized solution converges to an expanding state at an algebraic
rate in physical time. 
\end{abstract}
\noindent Key words: Navier--Stokes system; outer pressure; temperature-dependent viscosity; strong solution.

\section{Introduction} 
The one-dimensional motion of a viscous, heat-conducting polytropic gas is described in Lagrangian coordinates by the compressible Navier--Stokes system (see \cite{BGK,SJ})
\begin{equation}
	v_t=u_x,
	\label{equ-1}
\end{equation}
\begin{equation}
    u_t + \left(\frac{R\theta}{v}\right)_x = \left(\mu\frac{u_x}{v}\right)_x
    \label{equ-2},
\end{equation}
\begin{equation}
	\left(e+\frac{u^2}{2}\right)_t +\left(\frac{R\theta u}{v}\right)_x = \left(\kappa\frac{\theta_x}{v}+\mu\frac{uu_x}{v}\right)_x,
	\label{equ-3}
\end{equation}
Here $x\in[0,1]$ and $t\geq0$ denote space and time, respectively. The variables $v>0$, $u$, $\theta>0$ and $e$ denote the specific volume, velocity, absolute temperature and specific internal energy. The positive constant $R$ is the gas constant, and 
\begin{equation*}
	e=c_v\theta = \frac{R\theta}{\gamma-1},
\end{equation*} 
where $\gamma>1$ is the adiabatic exponent.

The viscosity coefficient $\mu$ and heat conductivity coefficient $\kappa$ depend on temperature and are given by
\begin{equation}
	\mu=\tilde{\mu} \theta^{\alpha}, \quad \kappa=\tilde{\kappa} \theta^{\beta} ,
	\label{mu-kappa}
\end{equation}
where $\tilde{\mu}$ and $\tilde{\kappa}$ are positive constants and $\alpha,\beta\geq 0 .$

The system (\ref{equ-1})-(\ref{mu-kappa}) is supplemented with the initial condition
\begin{equation}
	\label{equ-4}
	(v(x,0),u(x,0),\theta(x,0))=(v_0,u_0,\theta_0)(x), \quad x\in[0,1],
\end{equation}
and the outer pressure boundary conditions for $t>0$:
\begin{equation}
	\left(\frac{\mu}{v} u_{x}-R \frac{\theta}{v}\right)(0, t)=\left(\frac{\mu}{v} u_{x}-R \frac{\theta}{v}\right)(1, t)=-P(t),
	\label{equ-5}
\end{equation}
\begin{equation}
	\label{equ-6}
		\theta_{x}(0, t)=\theta_{x}(1, t)=0,
\end{equation}
where $t>0$.

Transport coefficients depending on density and temperature occur in the thermomechanical models studied in \cite{DMC,D-H}. For a dilute gas, the first Chapman--Enskog approximation gives viscosity and heat conductivity depending on temperature alone; see \cite{CS-CTG}.

For constant positive viscosity and heat conductivity, global existence and large-time behavior have been studied in \cite{L-L,A-Z,A-K-M,JS94,JS98,S-N,NT}. Uniform bounds for the specific volume were obtained by Jiang \cite{JS94,JS98}. For unbounded domains, Li and Liang \cite{L-L} established uniform temperature estimates and the large-time behavior of large solutions. 

For constant viscosity and temperature-dependent heat conductivity, Pan and Zhang \cite{P-Z} established global strong solutions. Related existence and asymptotic results can be found in \cite{H-S,J-K,WT,L-X,L-P-P}. When the viscosity coefficient $\mu$ depends on the specific volume $v$, it follows from (\ref{equ-1}) and (\ref{equ-2}) that
 \begin{equation}
	\label{mu-v}
	\left(\frac{\mu(v)v_x}{v}\right)_t=u_t+p_x,
\end{equation}
 Global solutions for more general thermomechanical models with $\mu=\mu(v)\geq C^{-1}$ and $\kappa=\kappa(v,\theta)\geq C^{-1}$ were studied in \cite{DMC,D-H,KB,WHD}. Duan, Guo and Zhu \cite{D-G-Z} proved global existence and uniqueness for $\mu=1+v^{-\alpha}$ and $\kappa=\theta^\beta$. 
 When $\mu$ is temperature-dependent, the identity corresponding to (\ref{mu-v}) becomes
 \begin{equation}
 	\label{mu-vthe}
 	\left(\frac{\mu(v,\theta)v_x}{v}\right)_t = u_t +p_x +\frac{\mu_\theta(v,\theta)}{v}\left(\theta_t v_x-u_x\theta_x\right).
 \end{equation}
 Wang and Zhao \cite{W-Z} considered the density-temperature dependent case where $\mu=\tilde{\mu}h(v)\theta^{\alpha}$, $\kappa = \tilde{\kappa}h(v)\theta^\alpha$, and showed the Cauchy problem for (\ref{equ-1})--(\ref{equ-3}) has a unique smooth solution with $|\alpha|\leq \varepsilon_0$. For related outer pressure problems with variable coefficients, see \cite{H-L,L97}.

When $\mu$ and $\kappa$ depend solely on $\theta$, \cite{L-Y-Z-Z,W-W} established a global existence result of Nishida-Smoller type. This result holds under the condition  that the adiabatic exponent $\gamma$ is sufficiently close to 1 and the initial temperature has small oscillation. Subsequently, Sun-Zhang-Zhao \cite{S-Z-Z} considered the case that $\mu=\theta^\alpha$  with $\alpha>0$  suitably small in (\ref{equ-4}),
requiring the initial data of $(v_0,u_0,\theta_0)$ to be in the space $ H^2\times H^2 \times H^2$. Thereafter, Dong-Tan\cite{D-T} relaxed the regularity requirement for the initial data to $H^2\times H^1\times H^1 $.

For the outer pressure problem, Nagasawa \cite{N88} established convergence to a stationary state when the pressure and transport coefficients are positive constants. Cai, Chen, and Peng \cite{C-C-P} investigated the same pressure condition but with parameters satisfying $\alpha=0,\beta>0$ in (\ref{mu-kappa}). The case of vanishing pressure $P(t)=0$ under the same parameter conditions was later considered by Cai et al. \cite{C-C-P-P}. They obtained decay of the normalized strong solutions, exponential in logarithmic time. Dong \cite{D25} considered more general asymptotic pressure profiles, including both $\lim _{t \rightarrow+\infty} P(t)>0$ and $\lim _{t \rightarrow+\infty} P(t)=0$, under the condition that
$$P(t)\in C^1[0,+\infty),\quad P^\prime(t) \in L^1\cap L^2(0,+\infty),$$ 
and established the large-time behavior for $H^2\times H^1\times H^1 $. 

We study the outer pressure problem (\ref{equ-1})--(\ref{equ-6}) with the temperature-dependent coefficients (\ref{mu-kappa}) and initial data in $H^1\times H^1\times H^1$. For each prescribed pressure in the classes below, we prove global existence and asymptotic behavior when the viscosity exponent is sufficiently small. The positive limiting pressure and vanishing limiting pressure cases are stated separately.

Throughout this paper $R,c_v,\tilde\mu,\tilde\kappa$ are fixed positive
constants. Subtracting \eqref{equ-2} multiplied by $u$ from \eqref{equ-3},
we obtain
\begin{equation}\label{internal-energy}
 c_v\theta_t+\frac{R\theta u_x}{v}
 =\left(\frac{\kappa\theta_x}{v}\right)_x+\frac{\mu u_x^2}{v}.
\end{equation}
\begin{theorem}\la{tm11} For some positive constants $M_{0}$,  $\underline{v}$, and $\underline{\theta}$, suppose that 
	the initial data $\left(v_{0}, u_{0}, \theta_{0}\right)$ satisfies
	\begin{equation}
		\inf_{x\in[0,1]}v_{0}(x) \geq \underline{v},\quad \inf_{x\in[0,1]}\theta_{0}(x) \geq \underline{\theta}, \quad \left\|\left(v_{0},  u_{0}, \theta_{0}\right)\right\|_{H^{1}} \leq M_{0},
		\label{th1-ini}
	\end{equation}	
	and the outer pressure $P(t)$  satisfies 
	\begin{equation}
        P(t)\geq0\ (t\geq0),\quad
        \lim_{t\to+\infty}P(t)=\bar P>0,\quad
        P\in C^1[0,+\infty),\quad P'\in L^1\cap L^2(0,+\infty).
		\label{th1-pt}
	\end{equation}
Then, for any $\beta\geq0$ and each fixed pressure $P$ satisfying
\eqref{th1-pt}, there is a positive
$\varepsilon_0=\varepsilon_0(M_0,\underline v,\underline\theta,
\beta,R,c_v,\tilde\mu,\tilde\kappa;P)$, depending only on the stated data and independent of the terminal time, such that problem (\ref{equ-1})--(\ref{mu-kappa})
with $0\leq\alpha\leq\varepsilon_0$ admits a unique global strong
solution $(v,u,\theta)$ satisfying
	\begin{equation}\la{th1-vu}
		\left\{\begin{array}{l}
			\left(v, u,  \theta\right) \in C\left([0, \infty); H^1(0,1)\right), \quad \left(u, \theta\right)\in C\left((0, \infty); H^2(0,1)\right), \\
			v_t \in C\left(0, \infty ; L^2(0,1)\right) \cap L^2\left(0, \infty ; H^1(0,1)\right), \\
			v_x, u_x, \theta_x, u_t, \theta_t, v_{x t}, u_{x x}, \theta_{x x} \in L^2((0,1) \times(0, \infty)),
		\end{array}\right.
	\end{equation}	
	and
	\begin{equation}
		\label{th1-tinft}
		\lim _{t \rightarrow+\infty}\|(v-\hat{v}, u-\hat{u}, \theta-\hat{\theta})\|_{H^{1}}=0,
	\end{equation}
	where
	\begin{equation}
		\begin{aligned}
			\label{th1-hat}
 &\hat u=\int_0^1u_0dx,\\
 &\hat\theta=\frac1{c_v+R}\left[\int_0^1\left(P(0)v_0+c_v\theta_0+
 \frac12u_0^2\right)dx-\frac12\hat u^2+
 \int_0^\infty P'(s)\int_0^1v(x,s)dx\,ds\right],\\
 &\hat v=\frac R{\bar P}\hat\theta.
		\end{aligned}
	\end{equation}
\end{theorem}

\begin{theorem}\la{tm12}
	Suppose that the initial data $(v_0, u_0, \theta_0)$ satisfies \eqref{th1-ini}, and the outer pressure $P(t)$ satisfying 
	\begin{equation}
        P(t)\geq0\ (t\geq0),\quad
        \lim_{t\to+\infty}P(t)=0,\quad
        P\in C^1[0,+\infty),\quad P'\in L^1\cap L^2(0,+\infty),
		\label{th1-pt11}
	\end{equation}
	and there exists a positive constant $\varepsilon \ll 1$ such that
	\begin{equation}
		\label{th2-pt}
		(1+t)^{\varepsilon} P(t),(1+t)^{3+\varepsilon}\left|P^{\prime}(t)\right|^{2} \in L^{1}(0,+\infty) .
	\end{equation}
Then, for any $\beta\geq0$ and each fixed pressure $P$ satisfying
\eqref{th1-pt11}--\eqref{th2-pt}, there is a number
$0<\varepsilon_0=\varepsilon_0(M_0,\underline v,\underline\theta,
\beta,R,c_v,\tilde\mu,\tilde\kappa;P)<1$, depending only on the stated data and independent of the terminal time, such that problem (\ref{equ-1})--(\ref{mu-kappa}) with
$0\leq\alpha\leq\varepsilon_0$ admits a unique global strong
solution $(v,u,\theta)$. The same global solution satisfies, for
every finite $T>0$,
\begin{equation}\label{th2-vu}
 \left\{\begin{aligned}
 &(v,u,\theta)\in C([0,T];H^1(0,1)),\qquad
       (u,\theta)\in C((0,T];H^2(0,1)),\\
 &v_t\in C((0,T];L^2(0,1))\cap L^2(0,T;H^1(0,1)),\\
 &v_x,u_x,\theta_x,u_t,\theta_t,v_{xt},u_{xx},\theta_{xx}
                          \in L^2((0,1)\times(0,T)).
 \end{aligned}\right.
\end{equation}
The global asymptotic conclusion is
	\begin{equation}
 \left\|\left(\frac{v}{1+t}-A(t),\,
 u-\int_0^1u_0dx-A(t)(x-\tfrac12),\,
 \theta-\left(\frac{\tilde\mu A(t)}R\right)^{1/(1-\alpha)}
                  \right)\right\|_{H^1}\leq C(1+t)^{-\lambda},
		\label{th2-contr}
	\end{equation}
	where $C$ and $\lambda$ are positive constants, and $A$ fulfills
	\begin{equation}
 \begin{aligned}
 c_v\left(\frac{\tilde\mu A(t)}R\right)^{1/(1-\alpha)}+\frac{A(t)^2}{24}
 &=\int_0^1\left(P(0)v_0+c_v\theta_0+\tfrac12u_0^2\right)dx
                  -\frac12\left(\int_0^1u_0dx\right)^2\\
 &\quad+\int_0^tP'(s)\int_0^1v(x,s)dx\,ds,
 \qquad A(t)>0.\label{th2-A}
 \end{aligned}
	\end{equation}
\end{theorem}

\begin{remark}
From a physical point of view, $P(t)$ is nonnegative; see \cite{N88}.
It may vanish on a finite initial interval in Theorem~\ref{tm11}.
For each prescribed pressure, the exponent threshold is chosen once
for all terminal times.
\end{remark}

\begin{remark}
Our result extends the results of Cai--Chen--Peng \cite{C-C-P} and
Cai et al.\ \cite{C-C-P-P}, where $\alpha=0$ and $\beta>0$.
\end{remark}

\begin{remark}
For the existence of global strong solutions, the initial regularity
$H^2\times H^1\times H^1$ in \cite{D25} is relaxed to
$H^1\times H^1\times H^1$.
\end{remark}

\begin{remark}
We consider a bounded interval. For unbounded domains with constant
viscosity $\mu=\tilde\mu$ and $P(t)\equiv R$, see \cite{H-W-Z,L-P-P}.
\end{remark}

We now make some comments on the analysis. The main difficulty is to
obtain uniform upper and lower bounds for $v$ and $\theta$ when
$\mu=\tilde\mu\theta^\alpha$. Following \cite{S-Z-Z}, we take
$\alpha$ sufficiently small to control the terms containing
$\mu_x$ and $\mu_t$. In the case $\bar P>0$, we use
\begin{equation*}
 \frac{\tilde u_t}{\mu}+\frac{R\theta_x}{\mu v}
       -\frac{R\theta v_x}{\mu v^2}
 =\frac{\mu_xu_x}{\mu v}+\left(\frac{v_x}{v}\right)_t.
\end{equation*}
In the case $\bar P=0$, after the change of variables in Section~3,
\begin{equation*}
 \left(\frac{v_x}{v}\right)_t+\frac{R\theta v_x}{\mu v^2}
 =\frac{w_t+w}{\mu}+\frac{R\theta_x}{\mu v}
       -\alpha\theta^{-1}\theta_x\left(1+\frac{w_x}{v}\right).
\end{equation*}
Multiplication by $v_x/v$ gives the estimates in Lemmas~\ref{lm25}
and \ref{lm34}. We then estimate $\theta^\beta\theta_x$ and obtain
the temperature bounds for all $\beta\geq0$. The smallness of
$\alpha$ depends on the initial bounds, the physical coefficients,
$\beta$ and the prescribed pressure. The higher order estimates
complete the continuation argument and the study of large-time behavior.

Sections~2 and~3 prove Theorems~\ref{tm11} and~\ref{tm12}, respectively.

\noindent\textbf{Notation.}
All spatial norms are over $(0,1)$. We write
\[
\|f\|_{L^p}=\|f\|_{L^p(0,1)},\qquad
\|(f_1,\ldots,f_n)\|=\sum_{j=1}^n\|f_j\|.
\]
We also use
$\sigma(t)=\min\{1,t\}$, and
$\tilde u=u-\int_0^1u_0dx$.
The constants $C$ may change from line to line and depend on the initial
 data, the prescribed pressure, $\beta$ and the fixed physical
coefficients, but are independent of the terminal time and of the
small exponent within its selected range. Finite-time bounds that also depend
on $T$ or on the temporary bound $M$ are written as $C(T)$ or $C(T,M)$.

\begin{pro}\label{local-property}
Let the positive initial data satisfy \eqref{th1-ini}, and let
$P\in C^1[0,\infty)$. There exists $T_0>0$ such that
\eqref{equ-1}--\eqref{mu-kappa} has a unique positive strong solution
on $[0,T_0]$ with the finite-time regularity in \eqref{th2-vu}.
The boundary conditions hold for $t>0$. The solution can be continued
while $v$ and $\theta$ stay in positive compact intervals and the
finite-time $H^1$ norms remain bounded.
\end{pro}
The local solution is obtained by the usual iteration and approximation
argument for one-dimensional viscous heat-conducting flow; see
\cite{A-K-M,D25} for the method. The estimates below are first proved for
smooth positive solutions and then passed to the strong solution by
approximation. No initial derivative trace is imposed on the $H^1$ data.
\section{Proof of Theorem 1.1}
For a finite $T>0$, let $(v,u,\theta)\in X(0,T;M)$, where
\begin{equation}\label{space}
\begin{aligned}
X(0,T;M)=\bigg\{&(v,u,\theta)\in C([0,T];H^1):\quad
 M^{-1}\leq v,\theta\leq M,\\
&\int_0^T\|(v_x,u_x,u_{xx},\theta_x,\theta_{xx},\theta_t)\|_{L^2}^2dt\leq M,\\
&\sup_{0\leq t\leq T}\|(v,u,\theta)(t)\|_{H^1}^2+
 \int_0^T\sigma(t)\|\theta_{xt}\|_{L^2}^2dt\leq M\bigg\}.
\end{aligned}
\end{equation}
The momentum equation and $v_{xt}=u_{xx}$ give
\begin{align*}
\int_0^T\|(u_t,v_{xt})\|_{L^2}^2dt
&\leq C(M)\int_0^T\left[\|u_{xx}\|_{L^2}^2+\|\theta_x\|_{L^2}^2
 +\|v_x\|_{L^2}^2\right.\\
&\hspace{26mm}\left.+(\|v_x\|_{L^2}^2+\|\theta_x\|_{L^2}^2)
                                    \|u_x\|_{L^\infty}^2\right]dt
 \leq C(M).
\end{align*}
In this section we assume
\begin{equation}\label{positive-smallness}
0\leq\alpha\leq\frac12,\qquad
\alpha M^{5/2}\leq1,\qquad \alpha M^{\beta+2}\leq1.
\end{equation}
In particular,
\begin{equation}\label{mu-range}
 |\alpha\log\theta|\leq\alpha\log M\leq1/e<\log2,
 \qquad \frac12\leq\theta^\alpha\leq2.
\end{equation}

\subsection{Lower order estimates}\label{21}
\begin{lemma}\label{lm21}
There exist $C>0$ and $0<c_1<c_2$, with $c_2\geq1$, such that
\begin{align}
&\sup_{0\leq t\leq T}\int_0^1
 \left[\frac12\tilde u^2+c_v(\theta-\log\theta-1)
                  +R(v-\log v-1)\right]dx
       +\int_0^T V(t)dt\leq C,\label{positive-entropy}\\
&c_1\leq\int_0^1v(x,t)dx,\ \int_0^1\theta(x,t)dx\leq c_2,
 \label{positive-means}
\end{align}
where the lower and upper bounds in \eqref{positive-means} apply to both means and
\begin{equation}\label{positive-dissipation}
V(t)=\int_0^1\left(\frac{\mu u_x^2}{v\theta}
                  +\frac{\kappa\theta_x^2}{v\theta^2}\right)dx.
\end{equation}
\end{lemma}
\begin{proof}
Integrating momentum and total energy yields
\begin{align}
&\int_0^1u(x,t)dx=\int_0^1u_0dx,\qquad
 \left(\int_0^x\tilde u(y,t)dy\right)_t
       =\frac{\mu u_x-R\theta}{v}+P(t),\label{integrated-momentum}\\
&\int_0^1\left(c_v\theta+\frac12\tilde u^2+P(t)v\right)dx\notag\\
 &=\int_0^1\left(c_v\theta_0+\frac12\tilde u_0^2+P(0)v_0\right)dx
   +\int_0^tP'(s)\int_0^1v(x,s)dx\,ds.\label{energy-work}
\end{align}
Choose $T_0\geq1$ so that $P(t)\geq\bar P/2$ for $t\geq T_0$.
For $t\leq\min\{T_0,T\}$, multiplication of
\eqref{integrated-momentum} by $v$ gives
\begin{align}
\tilde\mu\int_0^1\theta^\alpha v(x,t)dx
={}&\tilde\mu\int_0^1\theta_0^\alpha v_0dx
 -\int_0^1\left(\int_0^t\tilde u(x,s)ds\right)\tilde u(x,t)dx
 +\int_0^t\|\tilde u(s)\|_{L^2}^2ds\notag\\
&+\int_0^1v_0(x)\int_0^x(\tilde u(y,t)-\tilde u_0(y))dy\,dx
 +\int_0^t\int_0^1(R\theta-Pv)dxds\notag\\
&+\alpha\tilde\mu\int_0^t\int_0^1
                  \theta^{\alpha-1}\theta_s v\,dxds.\label{volume-work}
\end{align}
Here $\int_0^x\tilde u\,dy=0$ at $x=0,1$. For every $\eta>0$,
\begin{align*}
&\left|\int_0^1\left(\int_0^t\tilde u\,ds\right)\tilde u(t)dx\right|
 +\left|\int_0^1v_0\int_0^x\tilde u(y,t)dy\,dx\right|\\
&\quad\leq\left[\sqrt{T_0}\left(\int_0^t\|\tilde u\|_{L^2}^2ds\right)^{1/2}
                         +\|v_0\|_{L^1}\right]\|\tilde u(t)\|_{L^2}\\
&\quad\leq\eta\|\tilde u(t)\|_{L^2}^2
              +C_{T_0,\eta}\left(1+\int_0^t\|\tilde u\|_{L^2}^2ds\right),\\
&\alpha\tilde\mu\int_0^t\int_0^1
                       \theta^{\alpha-1}|\theta_s|v\,dxds
 \leq C\alpha M^2\sqrt{T_0}\left(\int_0^t\|\theta_s\|_{L^2}^2ds\right)^{1/2}
 \leq C\sqrt{T_0}.
\end{align*}
Consequently,
\begin{align}
\int_0^1v(t)dx
&\leq\eta\int_0^1(c_v\theta+\tfrac12\tilde u^2)(t)dx+C_{T_0,\eta}\notag\\
 &\quad+C_{T_0,\eta}\int_0^t\int_0^1(c_v\theta+\tfrac12\tilde u^2+v)dxds,
 \label{early-volume}\\
\int_0^1(c_v\theta+\tfrac12\tilde u^2+v)(t)dx
&\leq C_{T_0}+C_{T_0}\int_0^t(1+|P'(s)|)
                 \int_0^1(c_v\theta+\tfrac12\tilde u^2+v)dxds\notag\\
&\quad\leq C_{T_0}\exp\left\{C_{T_0}\left(T_0+\int_0^{T_0}|P'|ds\right)\right\}.
\label{early-energy}
\end{align}
In the second inequality choose $\eta$ with
$\eta(1+\|P\|_{L^\infty(0,T_0)})\leq1/2$ and use \eqref{energy-work}.

Dividing \eqref{internal-energy} by $\theta$ gives
\begin{align}
&\frac d{dt}\int_0^1(c_v\log\theta+R\log v)dx=V(t),\label{entropy-identity}\\
&\int_0^1\left[\tfrac12\tilde u^2+c_v(\theta-\log\theta-1)
                         +R(v-\log v-1)\right](t)dx+\int_0^tV(s)ds\notag\\
&=\int_0^1\left[\tfrac12\tilde u_0^2+c_v(\theta_0-\log\theta_0-1)
                         +R(v_0-\log v_0-1)\right]dx\notag\\
&\quad +(R-P(t))\int_0^1v(t)dx-(R-P(0))\int_0^1v_0dx
            +\int_0^tP'(s)\int_0^1v(s)dx\,ds\leq C_{T_0}.\label{early-entropy}
\end{align}
For $t\geq T_0$, subtracting \eqref{entropy-identity} from
\eqref{energy-work} differentiated in time yields
\begin{align}
&\frac d{dt}\int_0^1\left[\tfrac12\tilde u^2+c_v(\theta-\log\theta-1)
  +Pv-R-R\log\frac{Pv}{R}\right]dx+V(t)\notag\\
&\qquad=P'(t)\left(\int_0^1v\,dx-\frac R{P(t)}\right),\label{tail-entropy}\\
&\int_0^1v\,dx\leq\frac4{\bar P}\int_0^1
                      \left(Pv-R-R\log\frac{Pv}{R}+R\log2\right)dx.
\notag
\end{align}
Thus, for a fixed $C>0$,
\begin{align*}
&\int_0^1\left[\tfrac12\tilde u^2+c_v(\theta-\log\theta-1)
  +Pv-R-R\log\frac{Pv}{R}\right](t)dx+\int_{T_0}^tV(s)ds+C\\
&\quad\leq\left(C+C_{T_0}+|P(T_0)-R|\int_0^1v(T_0)dx+
                 R\left|\log\frac{P(T_0)}R\right|\right)\\
&\qquad\times\exp\left\{\frac4{\bar P}\int_{T_0}^t|P'(s)|ds\right\}\leq C.
\end{align*}
Since $\bar P/2\leq P\leq P(0)+\|P'\|_{L^1}$ on this interval,
\eqref{positive-entropy} follows from \eqref{early-entropy}.
Finally,
\begin{align*}
&z-1-\log z\geq\max\{z/2-\log2,-1-\log z\},\qquad z>0,\\
&\int_0^1\theta dx-1-\log\int_0^1\theta dx
 \leq\int_0^1(\theta-1-\log\theta)dx\leq C,\\
&\int_0^1v dx-1-\log\int_0^1v dx
 \leq\int_0^1(v-1-\log v)dx\leq C,
\end{align*}
which proves \eqref{positive-means}.
\end{proof}

\begin{lemma}\label{lm22}
The specific volume admits the representation
\begin{equation*}
v(x,t)=\frac1{\mathcal A(x,t)\mathcal B(x,t)}
 +\frac R{\tilde\mu}\int_0^t
 \frac{\mathcal A(x,s)\mathcal B(x,s)}{\mathcal A(x,t)\mathcal B(x,t)}
                         \theta(x,s)ds,
\end{equation*}
where
\begin{align}
\mathcal A(x,t)&=v_0(x)^{-\theta_0(x)^\alpha}
 v(x,t)^{\theta(x,t)^\alpha-1}
 \exp\left\{\frac1{\tilde\mu}\int_0^x(u_0-u)(y,t)dy\right\},\label{A}\\
\mathcal B(x,t)&=\exp\left\{\frac1{\tilde\mu}\int_0^tP(s)ds
 -\alpha\int_0^t(\theta^{\alpha-1}\theta_s\log v)(x,s)ds\right\}.
\label{B}
\end{align}
\end{lemma}
\begin{proof}
Integrating momentum and then differentiating the exponential gives
\begin{align}
&\tilde\mu(\theta^\alpha\log v)_t
 =\left(\int_0^x(u-u_0)dy\right)_t+\frac{R\theta}{v}-P
             +\alpha\tilde\mu\theta^{\alpha-1}\theta_t\log v,\notag\\
&v(x,t)=\frac1{\mathcal A(x,t)\mathcal B(x,t)}
 \exp\left\{\frac R{\tilde\mu}\int_0^t\frac{\theta}{v}(x,s)ds\right\}
 =\frac1{\mathcal A(x,t)\mathcal B(x,t)}\notag\\
&\hspace{32mm}+\frac R{\tilde\mu}\int_0^t
 \frac{\mathcal A(x,s)\mathcal B(x,s)}{\mathcal A(x,t)\mathcal B(x,t)}
                         \theta(x,s)ds.\label{volume-representation}
\end{align}
\end{proof}

\begin{lemma}\label{lm23}
There exists $C>1$ such that
\[
C^{-1}\leq v(x,t)\leq C,\qquad (x,t)\in[0,1]\times[0,T].
\]
\end{lemma}
\begin{proof}
{\it Step 1.} By \eqref{positive-smallness} and \eqref{positive-entropy},
\begin{align*}
&|(\theta^\alpha-1)\log v|\leq2\alpha(\log M)^2\leq C,
 \qquad C^{-1}\leq\mathcal A\leq C,\\
&\int_s^t\|\theta_t(\tau)\|_{L^\infty} d\tau
 \leq\sqrt{M(t-s)}+
 \sqrt{2M}\left(\int_s^t\sigma(\tau)^{-1/2}d\tau\right)^{1/2}
 \leq C\sqrt M(1+\sqrt{t-s}),\\
&\left|\alpha\int_s^t\theta^{\alpha-1}\theta_\tau\log v\,d\tau\right|
 \leq C\alpha M^{3/2}\log M(1+\sqrt{t-s})
 \leq C(1+\sqrt{t-s}),\\
&\int_s^tP(\tau)d\tau\geq\frac{\bar P}2(t-s)-C_{T_0},
 \qquad \int_s^tP(\tau)d\tau\leq\|P\|_{L^\infty}(t-s).
\end{align*}
Hence, with fixed positive $c,C$,
\begin{equation}\label{positive-kernel}
C^{-1}e^{-C(t-s)}\leq
 \frac{\mathcal A(x,s)\mathcal B(x,s)}{\mathcal A(x,t)\mathcal B(x,t)}
 \leq Ce^{-c(t-s)},\qquad 0\leq s\leq t\leq T.
\end{equation}
{\it Step 2.} Choose $y$ with $\theta(y,t)=\int_0^1\theta dx$. Then
\begin{align}
&\left|(\sqrt{\theta(x,t)}-\sqrt{c_1/2})_+
       -(\sqrt{\theta(y,t)}-\sqrt{c_1/2})_+\right|\notag\\
&\quad\leq\frac12\int_{\theta>c_1/2}\theta^{-1/2}|\theta_x|dx\notag\\
 &\quad\leq C V(t)^{1/2}\left(\int_{\theta>c_1/2}v\theta^{1-\beta}dx\right)^{1/2}
 \leq C\left(V(t)\max_xv(x,t)\right)^{1/2},\notag\\
&c-CV(t)\max_xv(x,t)\leq\theta(x,t)
                     \leq C+CV(t)\max_xv(x,t).\label{temperature-oscillation}
\end{align}
Substitution into \eqref{volume-representation} gives
\begin{align*}
\max_xv(x,t)&\leq C+C\int_0^tV(s)\max_xv(x,s)ds
 \leq C\exp\left\{C\int_0^TV(s)ds\right\}\leq C.
\end{align*}
{\it Step 3.} Fix $L\geq1$ so that $cL-C\int_0^TV(s)ds\geq cL/2$.
For $t\geq L$, \eqref{temperature-oscillation} and
\eqref{positive-kernel} imply
\begin{align*}
\int_{t-L}^t\theta(x,s)ds&\geq cL-C\int_{t-L}^tV(s)ds\geq cL/2,\\
v(x,t)&\geq C^{-1}e^{-CL}\int_{t-L}^t\theta(x,s)ds
                    \geq C^{-1}e^{-CL}cL/2.
\end{align*}
For $0\leq t\leq L$, the initial term in
\eqref{volume-representation} gives $v(x,t)\geq C^{-1}e^{-CL}\underline v$.
\end{proof}

\begin{lemma}\label{lm24}
There exists $c>0$ such that $\theta\geq c$. Moreover, for $0<\varepsilon<1$,
\begin{equation}\label{enhanced-dissipation}
\int_0^T\int_0^1\left(\theta^{\alpha-\varepsilon}u_x^2
                +\theta^{\beta-1-\varepsilon}\theta_x^2\right)dxdt
 \leq C_\varepsilon.
\end{equation}
\end{lemma}
\begin{proof}
Multiply \eqref{internal-energy} by
$-\theta^{-2}(\theta^{-1}-c_1^{-1})_+^p$, $p\geq3$.
Using $\theta_x=0$ at the endpoints, we obtain
\begin{align}
&\frac{c_v}{p+1}\frac d{dt}\int_0^1(\theta^{-1}-c_1^{-1})_+^{p+1}dx
 +\int_0^1\frac{\mu u_x^2}{v\theta^2}(\theta^{-1}-c_1^{-1})_+^pdx\notag\\
&\quad+2\tilde\kappa\int_0^1\frac{\theta^{\beta-3}\theta_x^2}{v}
                     (\theta^{-1}-c_1^{-1})_+^pdx
 +p\tilde\kappa\int_0^1\frac{\theta^{\beta-4}\theta_x^2}{v}
                     (\theta^{-1}-c_1^{-1})_+^{p-1}dx\notag\\
&=R\int_0^1\frac{u_x}{v\theta}(\theta^{-1}-c_1^{-1})_+^pdx
 \leq\frac12\int_0^1\frac{\mu u_x^2}{v\theta^2}
                      (\theta^{-1}-c_1^{-1})_+^pdx
       +C\int_0^1(\theta^{-1}-c_1^{-1})_+^pdx.\label{inverse-test}
\end{align}
At a mean-temperature point,
\begin{align*}
\|(\sqrt{c_1}-\sqrt\theta)_+\|_{L^\infty}^2
&\leq c_1^{-\beta}\|(c_1^{(\beta+1)/2}-\theta^{(\beta+1)/2})_+\|_{L^\infty}^2\\
&\leq C\left(\int_0^1\theta^{(\beta-1)/2}|\theta_x|dx\right)^2
 \leq CV(t)\int_0^1v\theta dx\leq CV(t),\\
(\theta^{-1}-c_1^{-1})_+^p
&\leq C(\sqrt{c_1}-\sqrt\theta)_+^2
                 \left[1+(\theta^{-1}-c_1^{-1})_+^{p+1}\right].
\end{align*}
The constant in the second inequality is independent of $p\geq3$.
Thus
\begin{align*}
&\frac d{dt}\left[1+\int_0^1(\theta^{-1}-c_1^{-1})_+^{p+1}dx\right]
 \leq C(p+1)V(t)\left[1+\int_0^1(\theta^{-1}-c_1^{-1})_+^{p+1}dx\right],\\
&\|(\theta^{-1}-c_1^{-1})_+(t)\|_{L^{p+1}}
 \leq\left[1+\int_0^1(\theta_0^{-1}-c_1^{-1})_+^{p+1}dx\right]^{1/(p+1)}
                      e^{C\int_0^TV(s)ds}\leq C.
\end{align*}
Letting $p\to\infty$ proves the lower bound.

Testing the heat equation by
$(\theta-2c_2)_+\theta^{-1-\varepsilon}$ yields
\begin{align}
&\int_0^1\frac{\mu u_x^2}{v}(\theta-2c_2)_+\theta^{-1-\varepsilon}dx
 +\varepsilon\tilde\kappa\int_{\theta>2c_2}
                   \frac{\theta^{\beta-1-\varepsilon}\theta_x^2}{v}dx\notag\\
&=c_v\frac d{dt}\int_{\theta>2c_2}\left[
 \frac{\theta^{1-\varepsilon}-(2c_2)^{1-\varepsilon}}{1-\varepsilon}
 +\frac{2c_2}{\varepsilon}(\theta^{-\varepsilon}-(2c_2)^{-\varepsilon})\right]dx\notag\\
&\quad +(1+\varepsilon)(2c_2)\tilde\kappa\int_{\theta>2c_2}
                  \frac{\theta^{\beta-2-\varepsilon}\theta_x^2}{v}dx
 +R\int_0^1\frac{u_x}{v}(\theta-2c_2)_+\theta^{-\varepsilon}dx.
\label{hot-test}
\end{align}
Moreover,
\begin{align*}
&\|(\theta-2c_2)_+\|_{L^\infty}
 \leq\frac12\left(\int_0^1\theta^{-1/2}|\theta_x|dx\right)^2
 \leq CV(t),\\
&R\left|\int_0^1\frac{u_x}{v}(\theta-2c_2)_+\theta^{-\varepsilon}dx\right|\\
 &\quad\leq\frac12\int_0^1\frac{\mu u_x^2}{v}
                         (\theta-2c_2)_+\theta^{-1-\varepsilon}dx
       +C_\varepsilon\|(\theta-2c_2)_+\|_{L^\infty}\int_0^1\theta dx,\\
&(1+\varepsilon)(2c_2)\tilde\kappa\int_{\theta>2c_2}
                     \frac{\theta^{\beta-2-\varepsilon}\theta_x^2}{v}dx
 \leq C_\varepsilon V(t),\\
&\theta^{\alpha-\varepsilon}\mathbf1_{\theta\leq4c_2}
 \leq(4c_2)^{1-\varepsilon}\theta^{\alpha-1},\\
 &\theta^{\beta-1-\varepsilon}\mathbf1_{\theta\leq 2c_2}
 \leq (2c_2)^{1-\varepsilon}\theta^{\beta-2}.
\end{align*}
Integrating \eqref{hot-test} gives \eqref{enhanced-dissipation}.
\end{proof}
\begin{lemma}\label{lm25}
It holds that
\begin{equation}\label{volume-gradient}
\sup_{t\leq T}\|v_x(t)\|_{L^2}^2+
 \int_0^T\int_0^1\theta v_x^2dxdt
 \leq\begin{cases}C,&\beta>0,\\ C_\varepsilon R_T^\varepsilon,&\beta=0,
 \end{cases}\qquad 0<\varepsilon<\tfrac12,
\end{equation}
where
\begin{equation}\label{RT}
 R_T=1+\sup_{(x,t)\in[0,1]\times[0,T]}\theta(x,t).
\end{equation}
\end{lemma}
\begin{proof}
By \eqref{equ-1} and \eqref{equ-2},
\begin{equation*}
 \left(\frac{v_x}{v}\right)_t+
 \frac{R\theta v_x}{\mu v^2}
 =\frac{\tilde u_t}{\mu}+\frac{R\theta_x}{\mu v}
                  -\frac{\alpha u_x\theta_x}{v\theta}.
\end{equation*}
Multiplying by $v_x/v$ and integrating over $[0,1]$, we obtain
\begin{align}
&\frac12\frac d{dt}\int_0^1\frac{v_x^2}{v^2}dx
       +\int_0^1\frac{R\theta v_x^2}{\mu v^3}dx\notag\\
&=\frac d{dt}\int_0^1\frac{\tilde u v_x}{\mu v}dx
 -\int_0^1\frac{\tilde u}{\mu}\left(\frac{v_x}{v}\right)_tdx
 +\alpha\int_0^1\frac{\tilde u v_x\theta_t}{\mu v\theta}dx
 +R\int_0^1\frac{v_x\theta_x}{\mu v^2}dx
 -\alpha\int_0^1\frac{u_xv_x\theta_x}{v^2\theta}dx\notag\\
&=\frac d{dt}\int_0^1\left(\frac{\tilde u v_x}{\mu v}
                                  -\frac{\tilde u^2}{2\mu^2}\right)dx
 -R\int_0^1\frac{\tilde u\theta_x}{\mu^2v}dx
 +R\int_0^1\frac{\theta\tilde u v_x}{\mu^2v^2}dx
 +R\int_0^1\frac{v_x\theta_x}{\mu v^2}dx\notag\\
&\quad+\alpha\int_0^1\frac1\theta
 \left(\frac{v_x}{v}-\frac{\tilde u}{\mu}\right)
 \left(\frac{\tilde u\theta_t}{\mu}-\frac{u_x\theta_x}{v}\right)dx\notag\\
&=\frac d{dt}\int_0^1\left(\frac{\tilde u v_x}{\mu v}
                       -\frac{\tilde u^2}{2\mu^2}\right)dx
                            +\sum_{j=1}^4J_j.
\label{effective-gradient-identity}
\end{align}
Since $\int_0^1\tilde u\,dx=0$, Lemmas~\ref{lm21}--\ref{lm24} give
\begin{align*}
\|\tilde u\|_{L^\infty}^2
&\leq\left(\int_0^1|u_x|dx\right)^2
 \leq\int_0^1\frac{\mu u_x^2}{v\theta}dx
             \int_0^1\frac{v\theta}{\mu}dx\leq CV(t),\\
|J_1|&\leq C\left(\int_0^1\theta\tilde u^2dx\right)^{1/2}
                  \left(\int_0^1\theta^{-1}\theta_x^2dx\right)^{1/2}
 \leq CV(t)+C\int_0^1\theta^{-1}\theta_x^2dx,\\
|J_2|&\leq\frac18\int_0^1\frac{R\theta v_x^2}{\mu v^3}dx
                  +C\int_0^1\theta\tilde u^2dx
 \leq\frac18\int_0^1\frac{R\theta v_x^2}{\mu v^3}dx+CV(t),\\
|J_3|&\leq\frac18\int_0^1\frac{R\theta v_x^2}{\mu v^3}dx
                                    +C\int_0^1\theta^{-1}\theta_x^2dx,\\
|J_4|&\leq\frac18\int_0^1\frac{R\theta v_x^2}{\mu v^3}dx
 +C\int_0^1\theta\tilde u^2dx
 +C\alpha^2\int_0^1\frac{\mu v}{R\theta^3}
           \left(\frac{\tilde u\theta_t}{\mu}-\frac{u_x\theta_x}{v}\right)^2dx.
\end{align*}
Furthermore,
\begin{align*}
&\alpha^2\int_0^T\int_0^1\frac{\mu v}{R\theta^3}
           \left(\frac{\tilde u\theta_t}{\mu}-\frac{u_x\theta_x}{v}\right)^2dxdt\\
&\quad\leq C\alpha^2\sup_t\|\tilde u\|_{L^\infty}^2
                            \int_0^T\|\theta_t\|_{L^2}^2dt
       +C\alpha^2\sup_t\|u_x\|_{L^2}^2
                            \int_0^T\|\theta_x\|_{L^\infty}^2dt
 \leq C\alpha^2M^2\leq C,\\
&\Phi(t)\triangleq\int_0^1\left(\frac{v_x^2}{2v^2}
      -\frac{\tilde u v_x}{\mu v}+\frac{\tilde u^2}{2\mu^2}\right)dx
 =\frac12\int_0^1\left(\frac{v_x}{v}-\frac{\tilde u}{\mu}\right)^2dx
 \geq\frac14\int_0^1\frac{v_x^2}{v^2}dx-C.
\end{align*}
Thus \eqref{effective-gradient-identity} gives
\begin{align*}
\Phi'(t)+c\int_0^1\theta v_x^2dx
&\leq CV(t)+C\int_0^1\theta^{-1}\theta_x^2dx
 +C\alpha^2\int_0^1\frac{\mu v}{R\theta^3}
     \left(\frac{\tilde u\theta_t}{\mu}-\frac{u_x\theta_x}{v}\right)^2dx.
\end{align*}
Integration over $[0,T]$ yields
\begin{align*}
\sup_{t\leq T}\|v_x(t)\|_{L^2}^2+
 \int_0^T\int_0^1\theta v_x^2dxdt
 &\leq C+C\int_0^T\int_0^1\theta^{-1}\theta_x^2dxdt.
\end{align*}
For $\beta>0$, choose $0<\varepsilon<\min\{\beta,1\}$ in
\eqref{enhanced-dissipation}. For $\beta=0$,
\[
 \int_0^T\int_0^1\theta^{-1}\theta_x^2dxdt
 \leq R_T^\varepsilon\int_0^T\int_0^1
                         \theta^{-1-\varepsilon}\theta_x^2dxdt
 \leq C_\varepsilon R_T^\varepsilon.
\]
This proves \eqref{volume-gradient}.
\end{proof}

\begin{lemma}\label{lm26}
It holds that
\begin{align}
&\int_0^T\|u_x\|_{L^2}^2dt\leq C,\label{velocity-first-integral}\\
&\sup_{t\leq T}\int_0^1\theta^{\beta+3}dx
 +\int_0^T\int_0^1\theta^{2\beta+1}\theta_x^2dxdt
\notag\\
&\quad\leq C_\eta+\eta\int_0^T\|u_{xx}\|_{L^2}^2dt
                  +C\int_0^T\|u_x\|_{L^2}^4dt,\qquad \eta>0.
\label{high-moment-bound}
\end{align}
\end{lemma}
\begin{proof}
Testing \eqref{internal-energy} by $(1-2c_2/\theta)_+$ gives
\begin{align*}
&\int_0^T\int_0^1\frac{\mu u_x^2}{v}(1-2c_2/\theta)_+dxdt\\
&\quad\leq C+C\int_0^T V(t)dt+
 C\int_0^T\int_0^1|u_x|(\theta-2c_2)_+dxdt\\
&\quad\leq C+\eta\int_0^T\|u_x\|_{L^2}^2dt+
 C_\eta\int_0^T\|(\theta-2c_2)_+\|_{L^\infty}\int_0^1\theta dxdt.
\end{align*}
The region $\theta\leq4c_2$ is controlled by
$u_x^2\leq C\mu u_x^2/(v\theta)$, and the last time integral is
bounded by $C\int_0^TV(t)dt$. Choosing $\eta$ small proves
\eqref{velocity-first-integral}.

Multiplication of
\eqref{internal-energy} by $(\theta^{\beta+2}-c_2^{\beta+2})_+$ yields
\begin{align}
&c_v\frac d{dt}\int_{\theta>c_2}
 \left[\frac{\theta^{\beta+3}}{\beta+3}-c_2^{\beta+2}\theta
                       +\frac{\beta+2}{\beta+3}c_2^{\beta+3}\right]dx
 +(\beta+2)\tilde\kappa\int_{\theta>c_2}
                       \frac{\theta^{2\beta+1}\theta_x^2}{v}dx\notag\\
&=\int_0^1\frac{\mu u_x^2-R\theta u_x}{v}
                         (\theta^{\beta+2}-c_2^{\beta+2})_+dx,
\label{high-moment-identity}\\
&\| (\theta^{\beta+2}-c_2^{\beta+2})_+\|_{L^\infty}
 \leq(\beta+2)\int_{\theta>c_2}\theta^{\beta+1}|\theta_x|dx
 \leq C\left(\int_{\theta>c_2}\theta^{2\beta+1}\theta_x^2dx\right)^{1/2}.
\notag
\end{align}
The right side of \eqref{high-moment-identity} is at most
\begin{align*}
&C(\|u_x\|_{L^2}^2+\|u_x\|_{L^\infty})
             \left(\int_{\theta>c_2}\theta^{2\beta+1}\theta_x^2dx\right)^{1/2}\\
&\quad\leq\frac{(\beta+2)\tilde\kappa}{2}
               \int_{\theta>c_2}\frac{\theta^{2\beta+1}\theta_x^2}{v}dx
 +\eta\|u_{xx}\|_{L^2}^2+C_\eta\|u_x\|_{L^2}^2+C\|u_x\|_{L^2}^4.
\end{align*}
Here
$\|u_x\|_{L^\infty}^2\leq C\|u_x\|_{L^2}^2+C\|u_x\|_{L^2}\|u_{xx}\|_{L^2}$.

The part $\theta\leq c_2$ is bounded by
$C_\varepsilon\int_0^T\int_0^1\theta^{\beta-1-\varepsilon}\theta_x^2dxdt$.
Integrating \eqref{high-moment-identity} and using
\eqref{velocity-first-integral}, we obtain \eqref{high-moment-bound}.
\end{proof}

\begin{lemma}\label{lm27}
For every $0<\varepsilon<1/2$,
\begin{align}
&\sup_{t\leq T}\left(\|u_x\|_{L^2}^2+\int_0^1\theta^{\beta+3}dx\right)\notag\\
 &\quad+\int_0^T\left(\|u_{xx}\|_{L^2}^2+\|u_t\|_{L^2}^2+
             \int_0^1\theta^{2\beta+1}\theta_x^2dx\right)dt
 \leq C_\varepsilon R_T ^{1+\varepsilon}.\label{velocity-moment}
\end{align}
\end{lemma}
\begin{proof}
At $x=0,1$, $u_x+(Pv-R\theta)/\mu=0$.
Testing momentum by $-\partial_x[u_x+(Pv-R\theta)/\mu]$, we obtain
\begin{align}
&\frac d{dt}\int_0^1\left[\frac12u_x^2+
                        \frac{Pv-R\theta}{\mu}u_x\right]dx
 +\int_0^1\frac\mu v u_{xx}^2dx\notag\\
&=\int_0^1u_x\left(\frac{Pv-R\theta}{\mu}\right)_tdx\notag\\
 &\quad-\int_0^1\left[\left(\frac\mu v\right)_xu_x-
                  \left(\frac{R\theta}v\right)_x\right]
       \left[u_{xx}+\left(\frac{Pv-R\theta}{\mu}\right)_x\right]dx\notag\\
 &\quad-\int_0^1\frac\mu v u_{xx}
                   \left(\frac{Pv-R\theta}{\mu}\right)_xdx.
\label{physical-stress-identity}
\end{align}
The derivatives in this identity are
\begin{align*}
\left(\frac{Pv-R\theta}{\mu}\right)_x
 &=\frac P\mu v_x-\left(\frac{\alpha Pv}{\mu\theta}
                                  +\frac{R(1-\alpha)}\mu\right)\theta_x,\\
\left(\frac{Pv-R\theta}{\mu}\right)_t
 &=\frac{P'v}\mu+\frac P\mu u_x-
       \frac{\alpha Pv}{\mu\theta}\theta_t-
       \frac{R(1-\alpha)}\mu\theta_t,\\
&-R(1-\alpha)\int_0^1\frac{u_x\theta_t}\mu dx\\
 &\quad=\frac{R\tilde\kappa(1-\alpha)}{c_v\tilde\mu}
                \int_0^1\frac{\theta^{\beta-\alpha}u_{xx}\theta_x}{v}dx\\
 &\qquad-\frac{R\tilde\kappa\alpha(1-\alpha)}{c_v\tilde\mu}
                \int_0^1\frac{\theta^{\beta-\alpha-1}u_x\theta_x^2}{v}dx\\
 &\qquad-\frac{R(1-\alpha)}{c_v}\int_0^1\frac{u_x^3}{v}dx
 +\frac{R^2(1-\alpha)}{c_v\tilde\mu}
                \int_0^1\frac{\theta^{1-\alpha}u_x^2}{v}dx.
\end{align*}
In particular, the product of the pressure gradient and the last part of
the spatial derivative contributes
$-R^2(1-\alpha)\int\theta^{-\alpha}\theta_x^2/(\tilde\mu v)dx$.
Expansion of \eqref{physical-stress-identity} now gives
\begin{align}
&\frac d{dt}\int_0^1\left[\frac12u_x^2+
                   \frac{Pv-R\theta}\mu u_x\right]dx
 +c(\|u_{xx}\|_{L^2}^2+\|\theta_x\|_{L^2}^2)\notag\\
&\leq C|P'|(1+\|u_x\|_{L^2}^2)
 +C\{\|u_x\|_{L^2}^2+\|u_x\|_{L^2}\|u_{xx}\|_{L^2}+
                         \|u_{xx}\|_{L^2}\|\theta^\beta\theta_x\|_{L^2}\}\notag\\
&\quad+C(\|u_x\|_{L^2}+\|u_x\|_{L^2}^{1/2}\|u_{xx}\|_{L^2}^{1/2})
\notag\\
&\qquad\times(\|u_x\|_{L^2}^2+\|v_x\|_{L^2}\|u_{xx}\|_{L^2}+
                       \|v_x\|_{L^2}\|\theta_x\|_{L^2}+\|v_x\|_{L^2}^2)\notag\\
&\quad+C\|\theta v_x\|_{L^2}(\|u_{xx}\|_{L^2}+\|\theta_x\|_{L^2})
                   +C\|\theta^{1/2}v_x\|_{L^2}^2\notag\\
&\quad+C\alpha(1+M^\beta)
 (\|u_x\|_{L^2}^2+\|\theta_t\|_{L^2}^2+\|\theta_x\|_{L^2}^2)\notag\\
&\quad+C\alpha(1+M^\beta)\|\theta_x\|_{L^2}
       (\|u_x\|_{L^2}+\|u_{xx}\|_{L^2})
                           (\|u_{xx}\|_{L^2}+\|\theta_x\|_{L^2}).
\label{physical-stress-products}
\end{align}
For the last line, the bounds in \eqref{space} give
\begin{align*}
&\alpha(1+M^\beta)\int_0^T
 \|\theta_x\|_{L^2}(\|u_x\|_{L^2}+\|u_{xx}\|_{L^2})
                    (\|u_{xx}\|_{L^2}+\|\theta_x\|_{L^2})dt\\
&\quad\leq\alpha(1+M^\beta)\sup_{t\leq T}\|\theta_x\|_{L^2}
 \left[\int_0^T(\|u_x\|_{L^2}+\|u_{xx}\|_{L^2})^2dt\right]^{1/2}\\
&\qquad\times
 \left[\int_0^T(\|u_{xx}\|_{L^2}+\|\theta_x\|_{L^2})^2dt\right]^{1/2}
 \leq C\alpha M^{\beta+3/2}\leq C,\\
&\alpha(1+M^\beta)\int_0^T
 (\|u_x\|_{L^2}^2+\|\theta_t\|_{L^2}^2+\|\theta_x\|_{L^2}^2)dt
 \leq C\alpha M^{\beta+1}\leq C.
\end{align*}
For the mixed terms,
\begin{align*}
\|v_x\|_{L^2}\|u_x\|_{L^2}^{1/2}\|u_{xx}\|_{L^2}^{3/2}
 &\leq\eta\|u_{xx}\|_{L^2}^2+C_\eta\|v_x\|_{L^2}^4\|u_x\|_{L^2}^2,\\
\|v_x\|_{L^2}\|\theta_x\|_{L^2}\|u_x\|_{L^2}^{1/2}\|u_{xx}\|_{L^2}^{1/2}
 &\leq\eta\|u_{xx}\|_{L^2}^2+
 C_\eta\|v_x\|_{L^2}^{4/3}\|\theta_x\|_{L^2}^{4/3}\|u_x\|_{L^2}^{2/3}\\
 &\leq\eta\|u_{xx}\|_{L^2}^2+
 C_\eta(\|v_x\|_{L^2}^4\|u_x\|_{L^2}^2+\|\theta^\beta\theta_x\|_{L^2}^2),\\
\|u_x\|_{L^2}^{5/2}\|u_{xx}\|_{L^2}^{1/2}
 &\leq\eta\|u_{xx}\|_{L^2}^2+C_\eta(\|u_x\|_{L^2}^2+\|u_x\|_{L^2}^4),\\
\|v_x\|_{L^2}^2\|u_x\|_{L^2}^{1/2}\|u_{xx}\|_{L^2}^{1/2}
 &\leq\eta\|u_{xx}\|_{L^2}^2+C_\eta(\|v_x\|_{L^2}^4+\|u_x\|_{L^2}^2),\\
\theta^{2\beta}&\leq\eta\theta^{2\beta+1}
                         +C_{\eta,\varepsilon}\theta^{\beta-1-\varepsilon}.
\end{align*}
Adding \eqref{high-moment-identity} to
\eqref{physical-stress-products}, and using
\eqref{volume-gradient} and \eqref{velocity-first-integral}, we find
\begin{align}
&\int_0^1\left[\frac12u_x^2+\frac{Pv-R\theta}{\mu}u_x\right](t)dx\notag\\
 &\quad+c_v\int_{\theta>c_2}\left(\frac{\theta^{\beta+3}}{\beta+3}-c_2^{\beta+2}\theta
       +\frac{\beta+2}{\beta+3}c_2^{\beta+3}\right)(t)dx\notag\\
&\quad+C R_T \notag\\
&\quad+c\int_0^t\left(\|u_{xx}\|_{L^2}^2+\|\theta_x\|_{L^2}^2
                +\int_0^1\theta^{2\beta+1}\theta_x^2dx\right)ds
 \leq C_\varepsilon R_T ^{1+\varepsilon}\notag\\
&\quad+C\int_0^t(|P'|+\|u_x\|_{L^2}^2)
 \left\{\int_0^1\left[\frac12u_x^2+\frac{Pv-R\theta}{\mu}u_x\right]dx
                    +C R_T \right.\notag\\
&\hspace{36mm}\left.+c_v\int_{\theta>c_2}
 \left(\frac{\theta^{\beta+3}}{\beta+3}-c_2^{\beta+2}\theta
       +\frac{\beta+2}{\beta+3}c_2^{\beta+3}\right)dx\right\}ds.
\label{moment-gronwall}
\end{align}
The expression in braces is at least
$1+\|u_x\|_{L^2}^2/4+c\int\theta^{\beta+3}dx$ after increasing $C$.
Gronwall's inequality proves the required bounds except for $u_t$.
Finally,
\begin{align*}
\|u_t\|_{L^2}^2
&\leq C(\|u_{xx}\|_{L^2}^2+\|\theta_x\|_{L^2}^2+\|\theta v_x\|_{L^2}^2)
\\
&\quad+C(\|v_x\|_{L^2}^2+\alpha^2\|\theta_x\|_{L^2}^2)
              (\|u_x\|_{L^2}^2+\|u_x\|_{L^2}\|u_{xx}\|_{L^2})\\
&\leq C(\|u_{xx}\|_{L^2}^2+\|\theta_x\|_{L^2}^2+\|\theta v_x\|_{L^2}^2)
\\
&\quad+C(1+\|v_x\|_{L^2}^4+\alpha^2M+\alpha^4M^2)\|u_x\|_{L^2}^2.
\end{align*}
Integration proves \eqref{velocity-moment}.
\end{proof}

\begin{lemma}\label{lm28}
The temperature and its derivatives satisfy
\begin{equation}\label{thermal-output}
\sup_{t\leq T}\left(\|\theta(t)\|_{L^\infty}+
                       \|\theta^\beta\theta_x(t)\|_{L^2}^2\right)
 +\int_0^T(\|\theta_t\|_{L^2}^2+\|\theta_{xx}\|_{L^2}^2)dt\leq C.
\end{equation}
\end{lemma}
\begin{proof}
Multiplying \eqref{internal-energy} by $\theta^\beta\theta_t$ gives
\begin{align}
&\frac{\tilde\kappa}{2}\frac d{dt}\int_0^1
                       \frac{\theta^{2\beta}\theta_x^2}{v}dx
 +c_v\int_0^1\theta^\beta\theta_t^2dx\notag\\
&=\int_0^1\frac{\theta^\beta}{v}(\mu u_x^2-R\theta u_x)\theta_tdx
 -\frac{\tilde\kappa}{2}\int_0^1
                       \frac{u_x\theta^{2\beta}\theta_x^2}{v^2}dx\notag\\
&\leq\frac{c_v}{2}\int_0^1\theta^\beta\theta_t^2dx
 +C\int_0^1(\theta^\beta u_x^4+\theta^{\beta+2}u_x^2)dx
 +C\|u_x\|_{H^1}\int_0^1\theta^{2\beta}\theta_x^2dx.
\label{thermal-gradient-identity}
\end{align}
Fix $0<\varepsilon<\min\{1, (\beta+1)/2\}$. The preceding lemmas imply
\begin{align*}
&\int_0^T\|u_x\|_{H^1}^2dt\leq C_\varepsilon R_T ^{1+\varepsilon},\\
&\int_0^T\int_0^1\theta^\beta u_x^4dxdt
 \leq C R_T ^\beta\sup_t\|u_x\|_{L^2}^2
                         \int_0^T\|u_x\|_{H^1}^2dt
 \leq C_\varepsilon R_T ^{\beta+2+2\varepsilon},\\
&\int_0^T\int_0^1\theta^{\beta+2}u_x^2dxdt
 \leq C R_T ^{\beta+2},\qquad
 \int_0^T\int_0^1\theta^{2\beta}\theta_x^2dxdt
 \leq C_\varepsilon R_T ^{\beta+1+\varepsilon},\\
&\|u_x\|_{H^1}\leq R_T ^{-1-\varepsilon}\|u_x\|_{H^1}^2
                             +\tfrac14 R_T ^{1+\varepsilon}.
\end{align*}
Substituting into \eqref{thermal-gradient-identity} and integrating,
\begin{align*}
&\frac{\tilde\kappa}{2}\int_0^1
             \frac{\theta^{2\beta}\theta_x^2}{v}(x,t)dx
 +\frac{c_v}{2}\int_0^t\int_0^1\theta^\beta\theta_t^2dxds\\
&\quad\leq C_\varepsilon R_T ^{\beta+2+2\varepsilon}
 +C R_T ^{-1-\varepsilon}\int_0^t\|u_x\|_{H^1}^2
          \int_0^1\frac{\theta^{2\beta}\theta_x^2}{v}dxds,\\
&  R_T ^{-1-\varepsilon}
                         \int_0^T\|u_x\|_{H^1}^2dt\leq C_\varepsilon,\\
&\sup_{t\leq T}\int_0^1\frac{\theta^{2\beta}\theta_x^2}{v}dx
 +\int_0^T\int_0^1\theta^\beta\theta_t^2dxdt\\
&\quad\leq C_\varepsilon R_T ^{\beta+2+2\varepsilon}
 \exp\left\{C R_T ^{-1-\varepsilon}
                              \int_0^T\|u_x\|_{H^1}^2dt\right\}.
\end{align*}
Consequently,
\begin{equation}\label{thermal-before-upper}
\sup_{t\leq T}\|\theta^\beta\theta_x\|_{L^2}^2
 +\int_0^T\|\theta^{\beta/2}\theta_t\|_{L^2}^2dt
 \leq C_\varepsilon R_T ^{\beta+2+2\varepsilon}.
\end{equation}
At a mean-temperature point,
\begin{align*}
\theta(x,t)^{\beta+3/2}
&\leq c_2^{\beta+3/2}+(\beta+3/2)
                       \int_0^1\theta^{\beta+1/2}|\theta_x|dx
 \leq C+C\|\theta^\beta\theta_x(t)\|_{L^2},\\
 R_T ^{\beta+3/2}
&\leq C_\varepsilon R_T ^{\beta/2+1+\varepsilon},
 \qquad  R_T ^{(\beta+1)/2-\varepsilon}\leq C_\varepsilon.
\end{align*}
Thus $\sup\theta\leq C$. Expanding the heat equation,
\begin{align}
\theta_{xx}
 &=\frac{c_vv}{\kappa}\theta_t-\frac\mu\kappa u_x^2
       +\frac{R\theta}{\kappa}u_x+\frac{v_x}{v}\theta_x
                                      -\frac\beta\theta\theta_x^2,
 \label{thermal-elliptic}\\
\|v_x\theta_x\|_{L^2}+\|\theta_x^2\|_{L^2}
 &\leq(\|v_x\|_{L^2}+\|\theta_x\|_{L^2})
       (\|\theta_x\|_{L^2}+\sqrt{2\|\theta_x\|_{L^2}\|\theta_{xx}\|_{L^2}})
\notag\\
 &\leq\eta\|\theta_{xx}\|_{L^2}+C_\eta\|\theta_x\|_{L^2},\notag\\
\|\theta_{xx}\|_{L^2}
 &\leq C(\|\theta_t\|_{L^2}+\|u_x\|_{H^1}+\|\theta_x\|_{L^2}).
 \notag
\end{align}
Squaring and integrating proves \eqref{thermal-output}.
\end{proof}

\subsection{Higher order estimates}
Lemmas~\ref{lm21}--\ref{lm28} imply the following corollary.
\begin{cor}\label{co29}
There exists $C>1$, independent of $T$, such that
\begin{align*}
&C^{-1}\leq v(x,t),\theta(x,t)\leq C,\\
&\sup_{t\leq T}\|(v_x,u_x,\theta_x)\|_{L^2}^2
 +\int_0^T\|(v_x,u_x,u_t,u_{xx},\theta_x,\theta_t,\theta_{xx})\|_{L^2}^2dt\leq C.
\end{align*}
\end{cor}
\begin{lemma}\label{lm29}
There is a constant $C>0$ such that
\begin{equation}\label{weighted-output}
\sup_{0<t\leq T}\sigma(t)\|(u_t,\theta_t,u_{xx},\theta_{xx})\|_{L^2}^2
 +\int_0^T\sigma(t)\|(u_{xt},\theta_{xt})\|_{L^2}^2dt\leq C.
\end{equation}
\end{lemma}
\begin{proof}
Differentiate the momentum and heat equations with respect to $t$.
The boundary conditions give
\begin{align}
&\frac12\frac d{dt}\|u_t\|_{L^2}^2+\int_0^1\frac\mu v u_{xt}^2dx
 =-P'(t)[u_t]_0^1\notag\\
&\quad-\int_0^1\left[
   \left(\frac{\alpha\mu u_x}{v\theta}-\frac Rv\right)\theta_t
                    -\frac{\mu u_x-R\theta}{v^2}u_x\right]u_{xt}dx,
 \label{momentum-time-identity}\\
&\frac{c_v}{2}\frac d{dt}\|\theta_t\|_{L^2}^2+
                       \int_0^1\frac\kappa v\theta_{xt}^2dx
 =-\int_0^1\frac\kappa v
                   \left(\frac{\beta\theta_t}\theta-\frac{u_x}v\right)
                                      \theta_x\theta_{xt}dx\notag\\
&\quad+\int_0^1\frac{\alpha\mu\theta^{-1}u_x^2-Ru_x}{v}\theta_t^2dx
 +\int_0^1\frac{2\mu u_x-R\theta}{v}u_{xt}\theta_tdx\notag\\
 &\quad-\int_0^1\frac{\mu u_x^2-R\theta u_x}{v^2}u_x\theta_tdx
 =\sum_{j=1}^4I_j.\label{thermal-time-identity}
\end{align}
By interval interpolation and the preceding lemmas,
\begin{align*}
&|P'[u_t]_0^1|\leq|P'|\|u_{xt}\|_{L^2}
                    \leq\eta\|u_{xt}\|_{L^2}^2+C_\eta|P'|^2,\\
&\|u_x\|_{L^\infty}^4
 \leq\bigl(\|u_x\|_{L^2}^2+2\|u_x\|_{L^2}\|u_{xx}\|_{L^2}\bigr)^2
 \leq C\|u_x\|_{H^1}^2,\\
&\|u_x^2\|_{L^2}^2+\|u_x^3\|_{L^2}^2
 \leq(\|u_x\|_{L^\infty}^2+\|u_x\|_{L^\infty}^4)\|u_x\|_{L^2}^2
 \leq C\|u_x\|_{H^1}^2,\\
&\left|\int_0^1\left[
    \left(\frac{\alpha\mu u_x}{v\theta}-\frac Rv\right)\theta_t
                  -\frac{\mu u_x-R\theta}{v^2}u_x\right]u_{xt}dx\right|\\
&\quad\leq C\left[(1+\alpha\|u_x\|_{L^\infty})\|\theta_t\|_{L^2}
                     +\|u_x^2\|_{L^2}+\|u_x\|_{L^2}\right]\|u_{xt}\|_{L^2}\\
&\quad\leq\eta\|u_{xt}\|_{L^2}^2
 +C_\eta(\|\theta_t\|_{L^2}^2+\|u_x\|_{H^1}^2)
 +C_\eta\alpha^2\|u_x\|_{H^1}^2\|\theta_t\|_{L^2}^2.
\end{align*}
For the four thermal terms,
\begin{align*}
|I_1|&\leq C\int_0^1(|\theta_t|+|u_x|)|\theta_x||\theta_{xt}|dx\\
&\leq C\left(\|\theta_x\|_{L^\infty}\|\theta_t\|_{L^2}
                +\|u_x\|_{L^\infty}\|\theta_x\|_{L^2}\right)\|\theta_{xt}\|_{L^2}\\
&\leq\eta\|\theta_{xt}\|_{L^2}^2
 +C_\eta\|\theta_x\|_{H^1}^2\|\theta_t\|_{L^2}^2
 +C_\eta\|u_x\|_{H^1}^2,\\
|I_2|&\leq C\int_0^1(u_x^2+|u_x|)\theta_t^2dx\\
&\leq C(\|u_x\|_{L^\infty}^2+\|u_x\|_{L^\infty})\|\theta_t\|_{L^2}^2
 \leq C(1+\|u_x\|_{H^1}^2)\|\theta_t\|_{L^2}^2,\\
|I_3|&\leq C(1+\|u_x\|_{L^\infty})\|u_{xt}\|_{L^2}\|\theta_t\|_{L^2}\\
&\leq\eta\|u_{xt}\|_{L^2}^2
       +C_\eta(1+\|u_x\|_{H^1}^2)\|\theta_t\|_{L^2}^2,\\
|I_4|&\leq C\int_0^1(|u_x|^3+u_x^2)|\theta_t|dx\\
&\leq C(\|u_x^3\|_{L^2}+\|u_x^2\|_{L^2})\|\theta_t\|_{L^2}
 \leq C\|u_x\|_{H^1}^2+C\|\theta_t\|_{L^2}^2.
\end{align*}
Using these estimates in \eqref{momentum-time-identity}--\eqref{thermal-time-identity},
and choosing $\eta$ first, we obtain
\begin{align}
&\frac d{dt}\left(\frac12\|u_t\|_{L^2}^2+\frac{c_v}{2}\|\theta_t\|_{L^2}^2\right)
              +c\|(u_{xt},\theta_{xt})\|_{L^2}^2\notag\\
&\quad\leq C(|P'|^2+\|u_x\|_{H^1}^2+\|\theta_x\|_{L^2}^2+\|\theta_t\|_{L^2}^2)
 +C(\|u_x\|_{H^1}^2+\|\theta_x\|_{H^1}^2)
               (\|u_t\|_{L^2}^2+\|\theta_t\|_{L^2}^2).\label{coupled-time}
\end{align}
By Corollary~\ref{co29} and \eqref{th1-pt},
\[
 \int_0^T\left(|P'|^2+\|u_x\|_{H^1}^2+\|\theta_x\|_{H^1}^2
                     +\|(u_t,\theta_t)\|_{L^2}^2\right)dt\leq C.
\]
There is therefore a sequence $s_j\downarrow0$ such that
$s_j\|(u_t,\theta_t)(s_j)\|_{L^2}^2\to0$.
Multiplication by $\sigma$, integration from $s_j$, and then $j\to\infty$ give
\begin{align*}
&\sup_{0<t\leq T}\sigma(t)\|(u_t,\theta_t)(t)\|_{L^2}^2
                  +\int_0^T\sigma\|(u_{xt},\theta_{xt})\|_{L^2}^2dt\\
&\quad\leq C\left[1+\int_0^{\min\{1,T\}}\|(u_t,\theta_t)\|_{L^2}^2dt\right]
 \exp\left\{C\int_0^T(\|u_x\|_{H^1}^2+\|\theta_x\|_{H^1}^2)dt\right\}
 \leq C.
\end{align*}
The derivative of the weight is
$\sigma'=\mathbf1_{(0,1)}$ almost everywhere. Finally,
\begin{align*}
u_{xx}&=\frac v\mu u_t-
          \left(\frac{\alpha\theta_x}\theta-\frac{v_x}v\right)u_x
               +\frac R\mu\theta_x-\frac{R\theta}{\mu v}v_x,\\
\|u_{xx}\|_{L^2}&\leq\tfrac12\|u_{xx}\|_{L^2}+
                C\|(u_t,u_x,v_x,\theta_x)\|_{L^2},\\
\|\theta_{xx}\|_{L^2}^2&\leq C(1+\|\theta_t\|_{L^2}^2+\|u_{xx}\|_{L^2}^2).
\end{align*}
Together with \eqref{thermal-elliptic}, these inequalities prove
\eqref{weighted-output}.
\end{proof}
\subsection{Stability of the solutions}
The preceding lemmas give constants $c,C>0$, independent of $M,T$,
such that
\begin{align}
&c\leq v,\theta\leq C,\qquad
 \sup_{t\leq T}\|(v,u,\theta)\|_{H^1}^2\notag\\
&\quad+\int_0^T\|(v_x,u_x,u_{xx},\theta_x,\theta_{xx},\theta_t)\|_{L^2}^2dt
            +\int_0^T\sigma\|\theta_{xt}\|_{L^2}^2dt\leq C.
\label{positive-complete-estimate}
\end{align}
The sum-norm convention uses
$(\sum_{j=1}^n a_j)^2\leq n\sum_{j=1}^na_j^2$ in this estimate.
Choose, in order,
\begin{equation}\label{positive-threshold}
M\geq4\max\{2,C,c^{-1},M_0^2,\underline v^{-1},\underline\theta^{-1}\},
\qquad
0<\varepsilon_0\leq\min\{\tfrac12,M^{-5/2},M^{-\beta-2}\}.
\end{equation}
For $0\leq\alpha\leq\varepsilon_0$, the estimates strictly improve
$X(0,T;M)$ to $X(0,T;M/2)$.
Continuity and Proposition~\ref{local-property} exclude a finite first
exit and give a solution on every finite interval. The constants in
\eqref{positive-threshold} depend on $P$ and are independent of $T$.
Increasing $T$ in \eqref{positive-complete-estimate} and
\eqref{weighted-output} proves \eqref{th1-vu}.

For $h=v_x,u_x,\theta_x$, respectively,
\begin{align}
&\int_1^\infty\left(\|h(t)\|_{L^2}^2+
                 \left|\frac d{dt}\|h(t)\|_{L^2}^2\right|\right)dt
 \leq\int_1^\infty\left(\|h\|_{L^2}^2+2\|h\|_{L^2}\|h_t\|_{L^2}\right)dt
 \leq C,\notag\\
&\lim_{t\to\infty}\|(v_x,u_x,\theta_x)(t)\|_{L^2}=0.
\label{positive-gradient-decay}
\end{align}
By \eqref{coupled-time} and \eqref{weighted-output},
\begin{align*}
&\int_1^\infty\left(\|u_t\|_{L^2}^2+c_v\|\theta_t\|_{L^2}^2\right)dt\leq C,\\
&\int_1^\infty\left[\frac d{dt}
 \left(\tfrac12\|u_t\|_{L^2}^2+\tfrac{c_v}{2}\|\theta_t\|_{L^2}^2\right)\right]_+dt\\
&\quad\leq C\int_1^\infty\left(|P'|^2+\|u_x\|_{H^1}^2
                 +\|\theta_x\|_{L^2}^2+\|\theta_t\|_{L^2}^2\right)dt\\
&\qquad+C\sup_{t\geq1}\|(u_t,\theta_t)\|_{L^2}^2
       \int_1^\infty\left(\|u_x\|_{H^1}^2+\|\theta_x\|_{H^1}^2\right)dt
 \leq C.
\end{align*}
The first integral and the positive-variation bound imply
\begin{align}
&\lim_{t\to\infty}(\|u_t(t)\|_{L^2}^2+c_v\|\theta_t(t)\|_{L^2}^2)=0,
 \qquad\lim_{t\to\infty}\|u_{xx}(t)\|_{L^2}=0,\notag\\
&\|u_x^2\|_{L^2}\leq\|u_x\|_{L^2}^2+
                \sqrt2\|u_x\|_{L^2}^{3/2}\|u_{xx}\|_{L^2}^{1/2}\longrightarrow0,
 \notag\\
&\|\theta_{xx}\|_{L^2}\leq C(\|\theta_t\|_{L^2}+\|u_x^2\|_{L^2}+
                               \|u_x\|_{L^2}+\|\theta_x\|_{L^2})\longrightarrow0.
\label{positive-second-decay}
\end{align}
For the mean values $\bar v=\int_0^1v dx$ and
$\bar\theta=\int_0^1\theta dx$, \eqref{integrated-momentum} gives
\begin{align}
R\bar\theta-P\bar v
 &=\int_0^1\mu u_xdx-\int_0^1v(x,t)\int_0^xu_t(y,t)dy\,dx,
 \label{mean-pressure-defect}\\
|R\bar\theta-P\bar v|&\leq C(\|u_x\|_{L^2}+\|u_t\|_{L^2}),\notag\\
\int_1^\infty\left|\frac d{dt}(R\bar\theta-P\bar v)^2\right|dt
 &\leq C\left\|R\bar\theta-P\bar v\right\|_{L^2(1,\infty)}
                 \|(\theta_t,u_x)\|_{L^2((0,1)\times(1,\infty))}
                  +C\|P'\|_{L^1}\leq C.\notag
\end{align}
It follows that $R\bar\theta-P\bar v\to0$. Meanwhile,
\begin{align*}
&(c_v+R)\bar\theta
  =\int_0^1\left(P(0)v_0+c_v\theta_0+\frac12u_0^2\right)dx
       -\frac12\int_0^1u^2dx\\
 &\qquad+\int_0^tP'(s)\bar v(s)ds+(R\bar\theta-P\bar v)\\
&\quad\longrightarrow
 \int_0^1\left(P(0)v_0+c_v\theta_0+\frac12u_0^2\right)dx
       -\frac12\hat u^2+\int_0^\infty P'(s)\bar v(s)ds
                    =(c_v+R)\hat\theta,\\
&\bar v\longrightarrow\frac{R\hat\theta}{\bar P}=\hat v,
 \qquad \hat\theta\geq c>0,\\
&\|v-\bar v\|_{H^1}+\|u-\hat u\|_{H^1}
                         +\|\theta-\bar\theta\|_{H^1}
 \leq C\|(v_x,u_x,\theta_x)\|_{L^2}\longrightarrow0.
\end{align*}
This proves \eqref{th1-tinft} and Theorem~\ref{tm11}.
\section{Proof of Theorem 1.2}
Use physical time $s$ and make the change of variables
\begin{equation}\label{time-change}
t=\log(1+s),\quad \check v(x,t)=e^{-t}v(x,e^t-1),\quad
\check u(x,t)=u(x,e^t-1),\quad \check\theta(x,t)=\theta(x,e^t-1).
\end{equation}
We omit the checks in this section. The equations become
\begin{equation}\label{normalized-system}
\left\{\begin{aligned}
&v_t+v=u_x,\qquad u_t=\left(\frac{\mu u_x-R\theta}{v}\right)_x,\\
&c_v\theta_t=\left(\frac{\kappa\theta_x}{v}\right)_x
                        +\frac{\mu u_x^2-R\theta u_x}{v},\\
&\left.\frac{\mu u_x-R\theta}{v}\right|_{x=0,1}
                    =-e^tP(e^t-1),\qquad \theta_x|_{x=0,1}=0.
\end{aligned}\right.
\end{equation}
Set
\begin{equation}\label{centered-w}
w=u-\int_0^1u_0dy-\int_0^xv(y,t)dy
                           +\int_0^1\int_0^xv(y,t)dy\,dx.
\end{equation}
Then
\begin{equation}\label{normalized-w-system}
\begin{gathered}
\int_0^1w dx=0,\qquad v_t=w_x,
 \qquad w_t+w=u_t=\left[\frac\mu v
                         \left(w_x+v-\frac{R\theta}\mu\right)\right]_x,\\
c_v\theta_t=\left(\frac{\kappa\theta_x}{v}\right)_x
 +\frac\mu v\left(w_x+v-\frac{R\theta}\mu\right)^2
 +\frac{R\theta}{v}\left(w_x+v-\frac{R\theta}\mu\right).
\end{gathered}
\end{equation}
For finite $T$ let $(v,u,w,\theta)\in Y(0,T;M)$, where
\begin{equation}\label{normalized-space}
\begin{aligned}
Y(0,T;M)=\bigg\{&(v,u,w,\theta)\in C([0,T];H^1),\quad
       M^{-1}\leq v,\theta\leq M,\\
&\sup_{t\leq T}\|(u,\theta,v)\|_{H^1}^2+
        \int_0^T\left\|w_x+v-\frac{R\theta}\mu\right\|_{L^1}dt\leq M,\\
&\int_0^T\|(v_x,w_x,\theta_x,\theta_t,\theta_{xx})\|_{L^2}^2dt
                     +\int_0^T\sigma\|\theta_{xt}\|_{L^2}^2dt\leq M\bigg\}.
\end{aligned}
\end{equation}
For the local solution, Proposition~\ref{local-property} and
\eqref{time-change} give the finite-time estimate
\[
 \int_0^T\|(u_t,u_{xx},w_t,w_{xx})\|_{L^2}^2dt\leq C(T,M).
\]
Until the parameter choice at the end, assume
\begin{equation}\label{normalized-smallness}
0\leq\alpha\leq\tfrac12,\qquad
\alpha(2M)^{5/2}\leq\eta_0,\qquad
\alpha(2M)^{\beta+2}\leq1,
\end{equation}
where $0<\eta_0\leq1$ is fixed in Lemma~\ref{lm32} independently of $M,T$.

Fix $0<\delta<\varepsilon/2$. By \eqref{th2-pt} and $P(s)\to0$,
\begin{align}
\int_0^\infty(1+s)^{1+\delta}|P'(s)|ds
&\leq\left(\int_0^\infty(1+s)^{3+\varepsilon}|P'(s)|^2ds\right)^{1/2}
                    (\varepsilon-2\delta)^{-1/2}\leq C,\notag\\
\sup_{s\geq0}(1+s)^{1+\delta}|P(s)|
&\leq\int_0^\infty(1+s)^{1+\delta}|P'(s)|ds\leq C.\notag
\end{align}
Changing variables by $s=e^t-1$, we obtain
\begin{align}
&\int_0^\infty e^{(1+\delta)t}P(e^t-1)dt
 =\int_0^\infty(1+s)^\delta P(s)ds\leq C,\notag\\
&\int_0^\infty e^{(2+\delta)t}|P'(e^t-1)|dt
 =\int_0^\infty(1+s)^{1+\delta}|P'(s)|ds\leq C,\notag\\
&\int_0^\infty e^{(4+2\delta)t}|P'(e^t-1)|^2dt
 =\int_0^\infty(1+s)^{3+2\delta}|P'(s)|^2ds\leq C,\notag\\
&\sup_{t\geq0}e^{(1+\delta)t}|P(e^t-1)|\leq C,\notag\\
&\int_0^\infty\left[e^{2t}|P(e^t-1)|^2+
              \left|\frac d{dt}(e^tP(e^t-1))\right|^2\right]dt\notag\\
&\quad\leq3\int_0^\infty e^{2t}|P(e^t-1)|^2dt
                  +2\int_0^\infty e^{4t}|P'(e^t-1)|^2dt\notag\\
&\quad\leq\frac3{2\delta}
      \left(\sup_{t\geq0}e^{(1+\delta)t}|P(e^t-1)|\right)^2
                  +2\int_0^\infty(1+s)^3|P'(s)|^2ds\leq C.
\label{pressure-weights}
\end{align}
Here $C$ also depends on the fixed $\delta$ and the weighted pressure
integrals in \eqref{th2-pt}.

\subsection{A priori estimates}
\begin{lemma}\label{lm31}
There exist $0<c_1<c_2$, with $c_2\geq1$, and $C>0$ such that
\begin{align}
&\sup_{t\leq T}\int_0^1
       \left[e^tP(e^t-1)v+u^2+\theta-\log\theta-1\right]dx
                         +\int_0^T\mathcal V(t)dt\leq C,\label{normalized-entropy}\\
&c_1\leq\int_0^1\theta(x,t)dx\leq c_2,
 \notag\\
 &\mathcal V(t)=\int_0^1\left[\frac\mu{v\theta}
             \left(w_x+v-\frac{R\theta}\mu\right)^2
                         +\frac\kappa{v\theta^2}\theta_x^2\right]dx.
\label{normalized-dissipation}
\end{align}
\end{lemma}
\begin{proof}
First return to physical time $s$ in \eqref{volume-work}.
The transformed temporary bounds give
\begin{align*}
&\left|\alpha\int_0^s\int_0^1
           \theta^{\alpha-1}\theta_\xi v\,dxd\xi\right|\\
&\quad\leq2\alpha M^2\int_0^{\log(1+s)}e^t\|\check\theta_t(t)\|_{L^2}dt\\
 &\quad\leq\sqrt2\alpha M^{5/2}\sqrt{(1+s)^2-1}\leq C(1+s),\\
&\int_0^1v(x,s)dx\\
&\quad\leq\eta(1+s)\|\tilde u(s)\|_{L^2}^2+
 C_\eta\left[1+s+\int_0^s\int_0^1
                    (\tilde u^2+\theta+|P|v)dxd\xi\right].
\end{align*}
Multiply the latter by $|P(s)|+(1+s)^{-1-\delta}$ and use
\eqref{energy-work}, choosing $\eta$ with
$C\eta\sup_s(1+s)(|P(s)|+(1+s)^{-1-\delta})\leq1/2$.
Then
\begin{align}
&\int_0^1\left[(|P(s)|+(1+s)^{-1-\delta})v+
                                      \tilde u^2+\theta\right]dx\notag\\
&\leq C+C(|P(s)|+(1+s)^{-1-\delta})
    \int_0^s\int_0^1\left[(|P(\xi)|+(1+\xi)^{-1-\delta})v+
                                     \tilde u^2+\theta\right]dxd\xi\notag\\
&\quad+C\int_0^s(1+\xi)^{1+\delta}|P'(\xi)|
      \int_0^1\left[(|P(\xi)|+(1+\xi)^{-1-\delta})v+
                                      \tilde u^2+\theta\right]dxd\xi.
\label{normalized-mean-volterra}
\end{align}
Applying the integrating factor first to the time integral in the
second line gives
\begin{align*}
&\int_0^s\int_0^1\left[(|P(\xi)|+(1+\xi)^{-1-\delta})v+
                                     \tilde u^2+\theta\right]dxd\xi\\
&\leq Cs\exp\left\{C\int_0^\infty(|P(\xi)|+(1+\xi)^{-1-\delta})d\xi\right\}\\
&\quad\times\left[1+\int_0^s(1+\xi)^{1+\delta}|P'(\xi)|
       \int_0^1\left[(|P(\xi)|+(1+\xi)^{-1-\delta})v+
                                    \tilde u^2+\theta\right]dxd\xi\right].
\end{align*}
Since $s(|P(s)|+(1+s)^{-1-\delta})\leq C$, substitution into
\eqref{normalized-mean-volterra} and a second Gronwall inequality yield
\begin{equation}\label{normalized-mean-energy}
\sup_{t\leq T}\int_0^1(u^2+\theta+e^tP(e^t-1)v)dx
 \leq C\exp\left\{C\int_0^\infty(1+s)^{1+\delta}|P'(s)|ds\right\}
 \leq C.
\end{equation}

Multiply \eqref{normalized-w-system}$_3$ by $-\theta^{\alpha-1}$,
and integrate momentum once in space. The two identities are
\begin{align*}
&-c_v\frac d{dt}\int_0^1\frac{\theta^\alpha-1}{\alpha}dx
 +\int_0^1\left[\frac{\tilde\mu\theta^{2\alpha-1}}v
       \left(w_x+v-\frac{R\theta}\mu\right)^2+
    (1-\alpha)\frac{\tilde\kappa\theta^{\beta+\alpha-2}\theta_x^2}{v}\right]dx\\
&\hspace{30mm}=-R\int_0^1\frac{\theta^\alpha}{v}
                         \left(w_x+v-\frac{R\theta}\mu\right)dx,\\
&\frac d{dt}\int_0^1\int_0^xu(y,t)dy\,dx
 =\tilde\mu\int_0^1\frac{\theta^\alpha}{v}
                     \left(w_x+v-\frac{R\theta}\mu\right)dx+e^tP(e^t-1).
\end{align*}
At $\alpha=0$, $(\theta^\alpha-1)/\alpha$ denotes $\log\theta$.
Adding the identities gives the exact cancellation
\begin{align}
&\frac d{dt}\left[-c_v\int_0^1\frac{\theta^\alpha-1}{\alpha}dx
                +\frac R{\tilde\mu}\int_0^1\int_0^xu\,dy\,dx\right]\notag\\
&\quad+\int_0^1\left[\frac{\tilde\mu\theta^{2\alpha-1}}v
          \left(w_x+v-\frac{R\theta}\mu\right)^2
  +(1-\alpha)\frac{\tilde\kappa\theta^{\beta+\alpha-2}\theta_x^2}{v}\right]dx
               =\frac R{\tilde\mu}e^tP(e^t-1).\label{normalized-alpha-entropy}
\end{align}
Using \eqref{normalized-mean-energy},
$|(\theta_0^\alpha-1)/\alpha|\leq
\max\{1,\sqrt{\theta_0}\}|\log\theta_0|$, and
$(z^\alpha-1)/\alpha\leq z-1$, we find
\begin{align*}
&-\int_0^1\frac{\theta^\alpha-1}\alpha dx\\
 &\quad+\frac1{c_v}\int_0^t\int_0^1\left[\frac{\tilde\mu\theta^{2\alpha-1}}v
                 \left(w_x+v-\frac{R\theta}\mu\right)^2
 +(1-\alpha)\frac{\tilde\kappa\theta^{\beta+\alpha-2}\theta_x^2}{v}\right]dxds
 \leq C,\\
&\frac{1-\theta^\alpha}\alpha
 =\int_0^{-\log\theta}e^{-\alpha z}dz
 \geq\tfrac12(-\log\theta),\qquad 0<\theta\leq1,\\
&\int_0^1(-\log\theta)_+dx\leq C,
 \\
 &\int_0^1\theta dx\geq\exp\left(\int_0^1\log\theta dx\right)\geq c_1>0.
\end{align*}
Finally $\theta^\alpha\geq1/2$ and $1-\alpha\geq1/2$ compare
the dissipation in \eqref{normalized-alpha-entropy} with
$\mathcal V/4$, proving the lemma.
\end{proof}

\begin{lemma}\label{lm32}
For a fixed $\eta_0>0$ in \eqref{normalized-smallness},
$C^{-1}\leq v\leq C$ on $[0,1]\times[0,T]$.
\end{lemma}
\begin{proof}
Define the integrating factor
\begin{equation}\label{normalized-factor}
\mathcal D(x,t)=\exp\left\{\frac{\theta(x,t)^{-\alpha}}{\tilde\mu}
       \int_0^x(u-u_0)(y,t)dy-
       \frac1{\tilde\mu}\int_0^te^sP(e^s-1)\theta(x,s)^{-\alpha}ds\right\}.
\end{equation}
Then $C^{-1}\leq\mathcal D\leq C$, and direct differentiation gives
\begin{align}
v(x,t)={}&e^{-(t-s_0)}\frac{\mathcal D(x,t)}{\mathcal D(x,s_0)}v(x,s_0)
 +\frac R{\tilde\mu}\int_{s_0}^t e^{-(t-s)}
             \frac{\mathcal D(x,t)}{\mathcal D(x,s)}\theta(x,s)^{1-\alpha}ds
 \notag\\
&+\frac\alpha{\tilde\mu}\int_{s_0}^te^{-(t-s)}
 \frac{\mathcal D(x,t)}{\mathcal D(x,s)}
          \theta^{-\alpha-1}\theta_s v\int_0^x(u-u_0)dy\,ds.
\label{normalized-representation}
\end{align}
The last integral has either sign. Its absolute value is bounded by
\begin{align}
C\alpha M^2\int_{s_0}^te^{-(t-s)}\|\theta_t(s)\|_{L^\infty} ds
&\leq C\alpha M^{5/2}
 \left[1+\left(\int_0^te^{-2(t-s)}\sigma(s)^{-1/2}ds\right)^{1/2}\right]
 \notag\\
&\leq C\alpha M^{5/2}\leq C\eta_0.\label{normalized-volume-error}
\end{align}
The temperature oscillation calculation in
\eqref{temperature-oscillation}, using only the mean bounds in
Lemma~\ref{lm31}, gives
\begin{equation}\label{normalized-oscillation}
c-C\mathcal V(t)\max_xv(x,t)\leq\theta(x,t)
             \leq C+C\mathcal V(t)\max_xv(x,t).
\end{equation}
With $s_0=0$ in \eqref{normalized-representation},
\begin{align*}
\max_xv(x,t)&\leq C+C\eta_0+
               C\int_0^t\mathcal V(s)\max_xv(x,s)ds
 \leq C\exp\left(C\int_0^T\mathcal V(s)ds\right)\leq C.
\end{align*}
This bound uses only $\eta_0\leq1$. Fix $L\geq1$ such that
$cL-C\int_0^T\mathcal V(s)ds\geq cL/2$.
Since $\theta^{1-\alpha}\geq\theta/2$, for $t\geq L$ take $s_0=t-L$ and retain the positive temperature source;
for $t\leq L$, retain the initial term. The resulting inequalities are
\begin{align*}
v(x,t)&\geq C^{-1}e^{-L}\int_{t-L}^t\theta(x,s)ds-C\eta_0
                 \geq C^{-1}e^{-L}cL/2-C\eta_0,
                      &&t\geq L,\\
v(x,t)&\geq C^{-1}e^{-L}\underline v-C\eta_0,
                      &&0\leq t\leq L.
\end{align*}
Choose $\eta_0$ after these fixed constants so that each right side is
at least half its positive first term. This proves the lemma.
\end{proof}

\begin{lemma}\label{lm33}
For every $0<\varepsilon<1$,
\begin{equation}\label{normalized-enhanced}
\theta\geq c>0,\qquad
\int_0^T\int_0^1\left[\theta^{\alpha-\varepsilon}
             \left(w_x+v-\frac{R\theta}\mu\right)^2
                      +\theta^{\beta-1-\varepsilon}\theta_x^2\right]dxdt
 \leq C_\varepsilon,\quad 0<\varepsilon<1.
\end{equation}
\end{lemma}
\begin{proof}
Apply the multipliers in \eqref{inverse-test} and \eqref{hot-test} to
\eqref{normalized-w-system}$_3$. For the inverse-temperature test, the
linear source satisfies
\begin{align*}
&-R\int_0^1\frac1{v\theta}
 \left(w_x+v-\frac{R\theta}\mu\right)(\theta^{-1}-c_1^{-1})_+^pdx\\
&\quad\leq\frac12\int_0^1\frac\mu{v\theta^2}
 \left(w_x+v-\frac{R\theta}\mu\right)^2
              (\theta^{-1}-c_1^{-1})_+^pdx
             +C\int_0^1(\theta^{-1}-c_1^{-1})_+^pdx.
\end{align*}
The mean bound in Lemma~\ref{lm31} and the volume bound in
Lemma~\ref{lm32} give the same cold-temperature oscillation estimate
as in Lemma~\ref{lm24}, with $V$ replaced by $\mathcal V$.
Consequently,
\[
 \| (\theta^{-1}-c_1^{-1})_+(t)\|_{L^{p+1}}
 \leq\left[1+\int_0^1(\theta_0^{-1}-c_1^{-1})_+^{p+1}dx\right]^{1/(p+1)}
             e^{C\int_0^T\mathcal V(s)ds}\leq C.
\]
Letting $p\to\infty$ proves $\theta\geq c$. For the hot-temperature test,
\begin{align*}
&R\left|\int_0^1\frac1v\left(w_x+v-\frac{R\theta}\mu\right)
                  (\theta-2c_2)_+\theta^{-\varepsilon}dx\right|\\
&\quad\leq\frac12\int_0^1\frac\mu v
 \left(w_x+v-\frac{R\theta}\mu\right)^2
                 (\theta-2c_2)_+\theta^{-1-\varepsilon}dx
 +C_\varepsilon\| (\theta-2c_2)_+\|_{L^\infty}\int_0^1\theta dx,\\
&\int_0^T\| (\theta-2c_2)_+\|_{L^\infty}\int_0^1\theta dxdt
 \leq C\int_0^T\mathcal V(t)dt\leq C.
\end{align*}
Integrating the identity corresponding to \eqref{hot-test} and using
the low-temperature comparisons there proves \eqref{normalized-enhanced}.
\end{proof}
\begin{lemma}\label{lm34}
It holds that
\begin{equation}\label{normalized-vx}
\sup_{t\leq T}\|v_x(t)\|_{L^2}^2
 \leq\begin{cases}C,&\beta>0,\\C_\varepsilon R_T^\varepsilon,&\beta=0,
 \end{cases}\qquad 0<\varepsilon<\tfrac12.
\end{equation}
Here $R_T$ is defined by \eqref{RT} for the solution of
\eqref{normalized-system}.
\end{lemma}
\begin{proof}
By \eqref{normalized-w-system},
\begin{align}
&\left(\frac{v_x}{v}\right)_t+\frac{R\theta v_x}{\mu v^2}
 =\frac{w_t+w}{\mu}+\frac{R\theta_x}{\mu v}
       -\frac\alpha\theta\left(1+\frac{w_x}{v}\right)\theta_x,
 \label{3.4-1}\\
&\left(\frac{v_x}{v}-\frac w\mu\right)_t
 +\frac{R\theta}{\mu v}\left(\frac{v_x}{v}-\frac w\mu\right)
 =\left(1-\frac{R\theta}{\mu v}\right)\frac w\mu
 +\frac{R\theta_x}{\mu v}
 +\frac\alpha\theta\left[\frac{w\theta_t}\mu
                    -\left(\frac{w_x}v+1\right)\theta_x\right].
\label{normalized-effective-gradient}
\end{align}
Multiplying \eqref{3.4-1} by $v_x/v$ and integrating over $[0,1]$,
we obtain
\begin{align}
&\frac12\frac d{dt}\int_0^1\frac{v_x^2}{v^2}dx
              +\int_0^1\frac{R\theta v_x^2}{\mu v^3}dx\notag\\
&=\frac d{dt}\int_0^1\left(\frac{wv_x}{\mu v}-\frac{w^2}{2\mu^2}\right)dx
 -\int_0^1\frac{w^2}{\mu^2}dx
 -R\int_0^1\frac{w\theta_x}{\mu^2v}dx
 +R\int_0^1\frac{\theta wv_x}{\mu^2v^2}dx\notag\\
&\quad+\int_0^1\frac{wv_x}{\mu v}dx
 +R\int_0^1\frac{v_x\theta_x}{\mu v^2}dx
 +\alpha\int_0^1\frac1\theta\left(\frac{v_x}{v}-\frac w\mu\right)
      \left[\frac{w\theta_t}\mu-\left(1+\frac{w_x}v\right)\theta_x\right]dx
 \notag\\
&=\frac d{dt}\int_0^1\left(\frac{wv_x}{\mu v}-\frac{w^2}{2\mu^2}\right)dx
                                        +\sum_{j=1}^6K_j.
\label{3.4-2}
\end{align}
It follows from Lemmas~\ref{lm31}--\ref{lm33} that
\begin{align*}
\sum_{j=1}^5|K_j|
&\leq\frac14\int_0^1\frac{R\theta v_x^2}{\mu v^3}dx
        +C\int_0^1(1+\theta)w^2dx
        +C\int_0^1\theta^{-1}\theta_x^2dx,\\
|K_6|&\leq\frac18\int_0^1\frac{R\theta v_x^2}{\mu v^3}dx
 +C\int_0^1\theta w^2dx
\\
&\quad+C\alpha^2\int_0^1\frac{\mu v}{R\theta^3}
       \left[\frac{w\theta_t}\mu-\left(1+\frac{w_x}v\right)\theta_x\right]^2dx,\\
&\alpha^2\int_0^T\int_0^1\frac{\mu v}{R\theta^3}
       \left[\frac{w\theta_t}\mu-\left(1+\frac{w_x}v\right)\theta_x\right]^2dxdt\\
&\quad\leq C\alpha^2\sup_t\|w\|_{L^\infty}^2\int_0^T\|\theta_t\|_{L^2}^2dt
 +C\alpha^2\sup_t\|w_x\|_{L^2}^2\int_0^T\|\theta_x\|_{L^\infty}^2dt\\
&\qquad+C\alpha^2\int_0^T\|\theta_x\|_{L^2}^2dt
 \leq C\alpha^2M^2,\\
&\|\theta\|_{L^\infty}\leq2c_2+C\mathcal V(t),\qquad
 \int_0^1(1+\theta)w^2dx\leq C(1+\mathcal V(t)).
\end{align*}
Let
\begin{equation*}
\tilde\Phi(t)\triangleq\int_0^1
 \left(\frac{v_x^2}{2v^2}+\frac{w^2}{2\mu^2}-\frac{wv_x}{\mu v}\right)dx
 =\frac12\int_0^1\left(\frac{v_x}{v}-\frac w\mu\right)^2dx.
\end{equation*}
Consequently,
\begin{align*}
&\tilde\Phi'(t)
 +c\int_0^1\left(\frac{v_x}{v}-\frac w\mu\right)^2dx\\
&\quad\leq C(1+\mathcal V(t))+C\int_0^1\theta^{-1}\theta_x^2dx
 +C\alpha^2\int_0^1\frac{\mu v}{R\theta^3}
      \left[\frac{w\theta_t}\mu-\left(1+\frac{w_x}v\right)\theta_x\right]^2dx,\\
&\left\|\frac{v_x(t)}{v(t)}-\frac{w(t)}{\mu(t)}\right\|_{L^2}^2
 \leq Ce^{-ct}+C\int_0^te^{-c(t-s)}
       \left(1+\mathcal V(s)+\int_0^1\theta^{-1}\theta_x^2dx\right)ds
                                     +C\alpha^2M^2.
\end{align*}
The argument following \eqref{effective-gradient-identity} and
$v_x=v(v_x/v-w/\mu)+vw/\mu$ now give \eqref{normalized-vx}.
\end{proof}

\begin{lemma}\label{lm36}
The following estimates hold:
\begin{align}
&\int_0^T\left\|u_x-\frac{R\theta}\mu\right\|_{L^2}^2dt\leq C,
 \label{normalized-first-square}\\
&\sup_{t\leq T}\left[\left\|u_x-\frac{R\theta}\mu\right\|_{L^2}^2
                       +\int_0^1\theta^{\beta+3}dx\right]\notag\\
 &\quad+\int_0^T\left[\left\|\left(u_x-\frac{R\theta}\mu\right)_x\right\|_{L^2}^2
               +\|u_t\|_{L^2}^2+\int_0^1\theta^{2\beta+1}\theta_x^2dx\right]dt
 \leq C_\varepsilon R_T ^{2\varepsilon}.\label{normalized-stress-output}
\end{align}
\end{lemma}
\begin{proof}
The test $(1-2c_2/\theta)_+$ in Lemma~\ref{lm26} applies to
\eqref{normalized-w-system}$_3$. The low-temperature part is controlled by
$\mathcal V$, while
\begin{align*}
&\int_0^T\int_0^1\frac\mu v
 \left(u_x-\frac{R\theta}\mu\right)^2(1-2c_2/\theta)_+dxdt\\
&\quad\leq C+\eta\int_0^T\left\|u_x-\frac{R\theta}\mu\right\|_{L^2}^2dt
 +C_\eta\int_0^T\| (\theta-2c_2)_+\|_{L^\infty}\int_0^1\theta dxdt\\
&\quad\leq C_\eta+\eta\int_0^T
                           \left\|u_x-\frac{R\theta}\mu\right\|_{L^2}^2dt.
\end{align*}
Choosing $\eta$ small proves \eqref{normalized-first-square}.
The boundary condition gives
\begin{equation}\label{corrected-residual-boundary}
\left.\left(u_x-\frac{R\theta}\mu+
                \frac{e^tP(e^t-1)v}\mu\right)\right|_{x=0,1}=0.
\end{equation}
Testing momentum with the negative spatial derivative of this expression gives
\begin{align}
&\frac12\frac d{dt}\left\|u_x-\frac{R\theta}\mu+
                      \frac{e^tP(e^t-1)v}\mu\right\|_{L^2}^2
 +\int_0^1\frac\mu v
       \left|\left(u_x-\frac{R\theta}\mu+
                      \frac{e^tP(e^t-1)v}\mu\right)_x\right|^2dx\notag\\
&=-\frac{R(1-\alpha)}{\tilde\mu}\int_0^1
       \left(u_x-\frac{R\theta}\mu+\frac{e^tP(e^t-1)v}\mu\right)
                         \theta^{-\alpha}\theta_tdx\notag\\
&\quad+\int_0^1\left(u_x-\frac{R\theta}\mu+
                    \frac{e^tP(e^t-1)v}\mu\right)
                    \left(\frac{e^tP(e^t-1)v}\mu\right)_tdx\notag\\
&\quad-\int_0^1\left(\frac\mu v\right)_x
       \left(u_x-\frac{R\theta}\mu+\frac{e^tP(e^t-1)v}\mu\right)
       \left(u_x-\frac{R\theta}\mu+\frac{e^tP(e^t-1)v}\mu\right)_xdx
 =\sum_{j=1}^3J_j.\label{normalized-stress-identity}
\end{align}
Substituting the heat equation into $J_1$ gives
\begin{align}
J_1={}&-\frac{R(1-\alpha)}{c_v}\int_0^1\frac1v
       \left(u_x-\frac{R\theta}\mu+\frac{e^tP(e^t-1)v}\mu\right)
                               \left(u_x-\frac{R\theta}\mu\right)^2dx
 \notag\\
&-\frac{R^2(1-\alpha)}{\tilde\mu c_v}\int_0^1
                 \frac{\theta^{1-\alpha}}v
                        \left(u_x-\frac{R\theta}\mu\right)^2dx\notag\\
 &-\frac{R^2(1-\alpha)e^tP(e^t-1)}{\tilde\mu^2c_v}
       \int_0^1\theta^{1-2\alpha}\left(u_x-\frac{R\theta}\mu\right)dx
 \notag\\
&+\frac{R\tilde\kappa(1-\alpha)}{\tilde\mu c_v}
   \int_0^1\frac{\theta^{\beta-\alpha}\theta_x}v
     \left(u_x-\frac{R\theta}\mu+\frac{e^tP(e^t-1)v}\mu\right)_xdx
 \notag\\
&-\frac{R\tilde\kappa\alpha(1-\alpha)}{\tilde\mu c_v}
   \int_0^1\frac{\theta^{\beta-\alpha-1}\theta_x^2}v
      \left(u_x-\frac{R\theta}\mu+\frac{e^tP(e^t-1)v}\mu\right)dx.
\label{normalized-stress-heat}
\end{align}
The last three terms satisfy, respectively,
\begin{align*}
&C|e^tP(e^t-1)|\left|\int_0^1\theta^{1-2\alpha}
                       \left(u_x-\frac{R\theta}\mu\right)dx\right|\\
&\quad\leq\frac{R^2(1-\alpha)}{2\tilde\mu c_v}
      \int_0^1\frac{\theta^{1-\alpha}}v
                     \left(u_x-\frac{R\theta}\mu\right)^2dx
                        +Ce^{2t}|P(e^t-1)|^2,\\
&C\left|\int_0^1\frac{\theta^{\beta-\alpha}\theta_x}v
       \left(u_x-\frac{R\theta}\mu+\frac{e^tP(e^t-1)v}\mu\right)_xdx\right|\\
&\quad\leq\eta\left\|\left(u_x-\frac{R\theta}\mu+
                  \frac{e^tP(e^t-1)v}\mu\right)_x\right\|_{L^2}^2
                         +C_\eta\|\theta^\beta\theta_x\|_{L^2}^2,\\
&C\alpha\left|\int_0^1\frac{\theta^{\beta-\alpha-1}\theta_x^2}v
            \left(u_x-\frac{R\theta}\mu+\frac{e^tP(e^t-1)v}\mu\right)dx\right|\\
&\quad\leq C\alpha M^\beta\|\theta_x\|_{L^2}^2
       \left\|\left(u_x-\frac{R\theta}\mu+\frac{e^tP(e^t-1)v}\mu\right)_x\right\|_{L^2}\\
&\quad\leq\eta\left\|\left(u_x-\frac{R\theta}\mu+
            \frac{e^tP(e^t-1)v}\mu\right)_x\right\|_{L^2}^2
                   +C_\eta\alpha^2M^{2\beta}\|\theta_x\|_{L^2}^4.
\end{align*}
Moreover,
\begin{align*}
&\alpha^2M^{2\beta}\int_0^T\|\theta_x\|_{L^2}^4dt
 \leq\alpha^2M^{2\beta}\sup_{t\leq T}\|\theta_x\|_{L^2}^2
                         \int_0^T\|\theta_x\|_{L^2}^2dt
 \leq\alpha^2M^{2\beta+2}\leq C,\\
&\left|\int_0^1\frac1v
 \left(u_x-\frac{R\theta}\mu+\frac{e^tP(e^t-1)v}\mu\right)
 \left(u_x-\frac{R\theta}\mu\right)^2dx\right|\\
&\quad\leq C\int_0^1
 \left|u_x-\frac{R\theta}\mu+\frac{e^tP(e^t-1)v}\mu\right|^3dx\\
 &\qquad+Ce^{2t}|P(e^t-1)|^2
       \left\|u_x-\frac{R\theta}\mu+\frac{e^tP(e^t-1)v}\mu\right\|_{L^1}\\
&\quad\leq C
 \left\|u_x-\frac{R\theta}\mu+\frac{e^tP(e^t-1)v}\mu\right\|_{L^2}^{5/2}
 \left\|\left(u_x-\frac{R\theta}\mu+
                      \frac{e^tP(e^t-1)v}\mu\right)_x\right\|_{L^2}^{1/2}\\
&\qquad+C\left\|u_x-\frac{R\theta}\mu+
                      \frac{e^tP(e^t-1)v}\mu\right\|_{L^2}^2
                 +Ce^{2t}|P(e^t-1)|^2\\
&\quad\leq\eta\left\|\left(u_x-\frac{R\theta}\mu+
              \frac{e^tP(e^t-1)v}\mu\right)_x\right\|_{L^2}^2\\
 &\qquad+C_\eta\left\|u_x-\frac{R\theta}\mu+
             \frac{e^tP(e^t-1)v}\mu\right\|_{L^2}^2\\
&\quad+C_\eta\left\|u_x-\frac{R\theta}\mu+
             \frac{e^tP(e^t-1)v}\mu\right\|_{L^2}^4
                       +C_\eta e^{2t}|P(e^t-1)|^2.
\end{align*}
For $J_2$ use the exact derivative
\begin{align*}
\left(\frac{e^tP(e^t-1)v}\mu\right)_t
={}&\frac v\mu\frac d{dt}(e^tP(e^t-1))
 +\frac{e^tP(e^t-1)}\mu
             \left(u_x-\frac{R\theta}\mu+\frac{R\theta}\mu-v\right)\\
 &-\frac{\alpha e^tP(e^t-1)v}{\mu\theta}\theta_t.
\end{align*}
Since $\int\theta^{1-2\alpha}dx\leq4c_2$,
\begin{align*}
|J_2|\leq{}&C\left|\frac d{dt}(e^tP(e^t-1))\right|
 \left\|u_x-\frac{R\theta}\mu+\frac{e^tP(e^t-1)v}\mu\right\|_{L^2}\\
&+C|e^tP(e^t-1)|
 \left\|u_x-\frac{R\theta}\mu+\frac{e^tP(e^t-1)v}\mu\right\|_{L^2}
                         \left\|u_x-\frac{R\theta}\mu\right\|_{L^2}\\
&+C|e^tP(e^t-1)|
 \left\|u_x-\frac{R\theta}\mu+\frac{e^tP(e^t-1)v}\mu\right\|_{L^\infty}
                      \int_0^1(1+\theta^{1-2\alpha})dx\\
&+C\alpha|e^tP(e^t-1)|
 \left\|u_x-\frac{R\theta}\mu+\frac{e^tP(e^t-1)v}\mu\right\|_{L^2}
                                      \|\theta_t\|_{L^2}\\
\leq{}&\eta\left\|\left(u_x-\frac{R\theta}\mu+
                 \frac{e^tP(e^t-1)v}\mu\right)_x\right\|_{L^2}^2
 +C_\eta\left\|u_x-\frac{R\theta}\mu+
                 \frac{e^tP(e^t-1)v}\mu\right\|_{L^2}^2\\
&+C_\eta\left[e^{2t}|P(e^t-1)|^2+
                  \left|\frac d{dt}(e^tP(e^t-1))\right|^2\right]\\
 &+C_\eta\alpha^2e^{2t}|P(e^t-1)|^2\|\theta_t\|_{L^2}^2,\\
|J_3|\leq{}&C(\|v_x\|_{L^2}+\alpha\|\theta_x\|_{L^2})
       \left\|u_x-\frac{R\theta}\mu+\frac{e^tP(e^t-1)v}\mu\right\|_{L^2}^{1/2}\\
 &\qquad\times\left\|\left(u_x-\frac{R\theta}\mu+
                    \frac{e^tP(e^t-1)v}\mu\right)_x\right\|_{L^2}^{3/2}\\
\leq{}&\eta\left\|\left(u_x-\frac{R\theta}\mu+
                   \frac{e^tP(e^t-1)v}\mu\right)_x\right\|_{L^2}^2\\
 &+C_{\eta,\varepsilon} R_T ^{2\varepsilon}
       \left\|u_x-\frac{R\theta}\mu+\frac{e^tP(e^t-1)v}\mu\right\|_{L^2}^2.
\end{align*}
By \eqref{normalized-first-square} and \eqref{pressure-weights},
\begin{align*}
& R_T ^{2\varepsilon}\int_0^T
 \left\|u_x-\frac{R\theta}\mu+\frac{e^tP(e^t-1)v}\mu\right\|_{L^2}^2dt\\
&\quad\leq C R_T ^{2\varepsilon}
 \int_0^T\left[\left\|u_x-\frac{R\theta}\mu\right\|_{L^2}^2
                           +e^{2t}|P(e^t-1)|^2\right]dt
 \leq C R_T ^{2\varepsilon}.
\end{align*}

The high-temperature multiplier in \eqref{high-moment-identity} gives
\begin{align}
&c_v\frac d{dt}\int_{\theta>c_2}\left[
 \frac{\theta^{\beta+3}}{\beta+3}-c_2^{\beta+2}\theta
                  +\frac{\beta+2}{\beta+3}c_2^{\beta+3}\right]dx
               +c\int_0^1\theta^{2\beta+1}\theta_x^2dx\notag\\
&\quad\leq\eta\| (\mu u_x-R\theta)_x\|_{L^2}^2
 +C_\eta\left\|u_x-\frac{R\theta}\mu\right\|_{L^2}^2
 +C\left\|u_x-\frac{R\theta}\mu\right\|_{L^2}^4
 +C_\varepsilon\int_0^1\theta^{\beta-1-\varepsilon}\theta_x^2dx.
\label{normalized-high-moment}
\end{align}
Indeed, $\int\theta^{1-\alpha}dx\leq2c_2$ and
\begin{align*}
&\left|\int_0^1\frac1v
 \left[\mu\left(u_x-\frac{R\theta}\mu\right)^2
             +R\theta\left(u_x-\frac{R\theta}\mu\right)\right]
                        (\theta^{\beta+2}-c_2^{\beta+2})_+dx\right|\\
&\quad\leq C\left[\left\|u_x-\frac{R\theta}\mu\right\|_{L^2}^2
                     +\|\mu u_x-R\theta\|_{L^\infty}\right]
              \left(\int_{\theta>c_2}\theta^{2\beta+1}\theta_x^2dx\right)^{1/2}\\
&\quad\leq\frac{(\beta+2)\tilde\kappa}{2}
             \int_{\theta>c_2}\frac{\theta^{2\beta+1}\theta_x^2}{v}dx
 +C\left\|u_x-\frac{R\theta}\mu\right\|_{L^2}^4
 +C\|\mu u_x-R\theta\|_{L^\infty}^2,\\
&\|\mu u_x-R\theta\|_{L^\infty}^2
 \leq\|\mu u_x-R\theta\|_{L^2}^2
       +2\|\mu u_x-R\theta\|_{L^2}\|(\mu u_x-R\theta)_x\|_{L^2}\\
&\quad\leq\eta\|(\mu u_x-R\theta)_x\|_{L^2}^2
                         +C_\eta\left\|u_x-\frac{R\theta}\mu\right\|_{L^2}^2.
\end{align*}
Furthermore,
\begin{align*}
&(\mu u_x-R\theta)_x\\
&\quad=\mu\left(u_x-\frac{R\theta}\mu+
                 \frac{e^tP(e^t-1)v}\mu\right)_x\\
 &\qquad+\frac{\alpha\mu\theta_x}\theta
       \left(u_x-\frac{R\theta}\mu+\frac{e^tP(e^t-1)v}\mu\right)
 -e^tP(e^t-1)v_x,\\
&\alpha\|\theta_x\|_{L^2}
 \left\|u_x-\frac{R\theta}\mu+\frac{e^tP(e^t-1)v}\mu\right\|_{L^\infty}\\
&\quad\leq C\alpha\sqrt M
 \left\|\left(u_x-\frac{R\theta}\mu+
                   \frac{e^tP(e^t-1)v}\mu\right)_x\right\|_{L^2},\\
&\|(\mu u_x-R\theta)_x\|_{L^2}^2\\
&\quad\leq C\left\|\left(u_x-\frac{R\theta}\mu+
                 \frac{e^tP(e^t-1)v}\mu\right)_x\right\|_{L^2}^2
                           +Ce^{2t}|P(e^t-1)|^2\|v_x\|_{L^2}^2,\\
&\theta^{2\beta}\leq\zeta\theta^{2\beta+1}
                      +C_{\zeta,\varepsilon}\theta^{\beta-1-\varepsilon}.
\end{align*}
Choose $\eta$ first and then $\zeta$ to add
\eqref{normalized-high-moment} to \eqref{normalized-stress-identity}.
Using \eqref{normalized-first-square} and \eqref{pressure-weights}, we obtain
\begin{align}
&\left\|u_x-\frac{R\theta}\mu+\frac{e^tP(e^t-1)v}\mu\right\|_{L^2}^2
 +\int_0^1\theta^{\beta+3}(x,t)dx\notag\\
&\quad+c\int_0^t\left[
 \left\|\left(u_x-\frac{R\theta}\mu+
                 \frac{e^sP(e^s-1)v}\mu\right)_x\right\|_{L^2}^2
               +\int_0^1\theta^{2\beta+1}\theta_x^2dx\right]ds\notag\\
&\leq C_\varepsilon R_T ^{2\varepsilon}
 +C\int_0^t\left[
 \left\|u_x-\frac{R\theta}\mu+\frac{e^sP(e^s-1)v}\mu\right\|_{L^2}^2
                          +e^{2s}|P(e^s-1)|^2\right]\notag\\
&\qquad\qquad\times\left[
 \left\|u_x-\frac{R\theta}\mu+\frac{e^sP(e^s-1)v}\mu\right\|_{L^2}^2
                   +\int_0^1\theta^{\beta+3}dx\right]ds,
\label{normalized-stress-combination}\\
&\int_0^T\left[
 \left\|u_x-\frac{R\theta}\mu+\frac{e^tP(e^t-1)v}\mu\right\|_{L^2}^2
                            +e^{2t}|P(e^t-1)|^2\right]dt\leq C.
\notag
\end{align}
Gronwall's inequality yields
\begin{align*}
&\sup_{t\leq T}\left[
 \left\|u_x-\frac{R\theta}\mu+\frac{e^tP(e^t-1)v}\mu\right\|_{L^2}^2
                       +\int_0^1\theta^{\beta+3}dx\right]\\
&\quad+\int_0^T\left[
 \left\|\left(u_x-\frac{R\theta}\mu+
                \frac{e^tP(e^t-1)v}\mu\right)_x\right\|_{L^2}^2
                  +\int_0^1\theta^{2\beta+1}\theta_x^2dx\right]dt\\
&\leq C_\varepsilon R_T ^{2\varepsilon}
 \exp\left\{C\int_0^T\left[
 \left\|u_x-\frac{R\theta}\mu+\frac{e^tP(e^t-1)v}\mu\right\|_{L^2}^2
                  +e^{2t}|P(e^t-1)|^2\right]dt\right\}\\
&\leq C_\varepsilon R_T ^{2\varepsilon}.
\end{align*}
The remaining derivatives follow from
\begin{align*}
\left(\frac{e^tP(e^t-1)v}\mu\right)_x
 &=\frac{e^tP(e^t-1)}\mu v_x-
          \frac{\alpha e^tP(e^t-1)v}{\mu\theta}\theta_x,\\
u_t&=\left[\frac\mu v\left(u_x-\frac{R\theta}\mu+
                                \frac{e^tP(e^t-1)v}\mu\right)\right]_x.
\end{align*}
\end{proof}

\begin{lemma}\label{lm37}
There exist fixed $c,C>0$ such that
\begin{align}
&c\leq v,\theta\leq C,\quad
 \sup_{t\leq T}\left(\|v_x\|_{L^2}^2+\|\theta_x\|_{L^2}^2+
                       \left\|u_x-\frac{R\theta}\mu\right\|_{L^2}^2\right)\leq C,
 \notag\\
&\int_0^T\left(\left\|u_x-\frac{R\theta}\mu\right\|_{H^1}^2
        +\|u_t\|_{L^2}^2+\|\theta_x\|_{H^1}^2+\|\theta_t\|_{L^2}^2\right)dt\leq C.
\label{normalized-base}
\end{align}
\end{lemma}
\begin{proof}
The thermal-gradient identity is now
\begin{align}
&\frac{\tilde\kappa}{2}\frac d{dt}\int_0^1
              \frac{\theta^{2\beta}\theta_x^2}{v}dx+c_v\int_0^1\theta^\beta\theta_t^2dx
 \notag\\
&=\int_0^1\frac{\theta^\beta\theta_t}{v}
      \left[\mu\left(u_x-\frac{R\theta}\mu\right)^2+
                      R\theta\left(u_x-\frac{R\theta}\mu\right)\right]dx
 -\frac{\tilde\kappa}{2}\int_0^1\frac{(u_x-v)\theta^{2\beta}\theta_x^2}{v^2}dx.
\label{normalized-thermal-identity}
\end{align}
Here
$|u_x-v|\leq|u_x-R\theta/\mu|+C(1+\theta)$.
The calculation in \eqref{thermal-gradient-identity}--\eqref{thermal-before-upper}
therefore applies with $u_x$ replaced by $u_x-R\theta/\mu$ and the additional bound
\begin{align*}
\int_0^T\int_0^1(1+\theta)\theta^{2\beta}\theta_x^2dxdt
&\leq C R_T ^{\beta+2+\varepsilon}
       \int_0^T\int_0^1\theta^{\beta-1-\varepsilon}\theta_x^2dxdt\\
&\leq C_\varepsilon R_T ^{\beta+2+\varepsilon}.
\end{align*}
Taking $\varepsilon=1/8$ gives
\begin{align*}
\sup_t\|\theta^\beta\theta_x\|_{L^2}^2+
                 \int_0^T\|\theta^{\beta/2}\theta_t\|_{L^2}^2dt
 &\leq C R_T ^{\beta+2+2\varepsilon},\\
 R_T ^{\beta+3/2}
 &\leq C R_T ^{\beta/2+1+\varepsilon},
 \qquad \sup\theta\leq C,\\
\theta_{xx}
 &=\frac{c_vv}{\kappa}\theta_t
   -\frac\mu\kappa\left(u_x-\frac{R\theta}\mu\right)^2
   -\frac{R\theta}{\kappa}\left(u_x-\frac{R\theta}\mu\right)
   \\
 &\quad+\frac{v_x}v\theta_x-\frac\beta\theta\theta_x^2.
\end{align*}
The last identity and the same two-product absorption as in
\eqref{thermal-elliptic} prove \eqref{normalized-base}.
\end{proof}

\subsection{Stability of the solutions}
\begin{cor}\label{co3.1}
Under the assumptions of this section,
\begin{align*}
&C^{-1}\leq v,\theta\leq C,\qquad
 \sup_{t\leq T}\|(v_x,u,u_x,w,w_x,\theta_x)\|_{L^2}^2\leq C,\\
&\int_0^T\left[\|(u_t,u_{xx},\theta_x,\theta_t,\theta_{xx})\|_{L^2}^2
 +\left\|u_x-\frac{R\theta}\mu\right\|_{H^1}^2\right]dt\leq C.
\end{align*}
\end{cor}
\begin{lemma}\label{lm39}
The following estimates hold:
\begin{align}
&\sup_{0<t\leq T}\sigma(t)
 \left[\|(u_t,\theta_t,u_{xx},\theta_{xx})\|_{L^2}^2+
                \left\|\left(u_x-\frac{R\theta}\mu\right)_x\right\|_{L^2}^2\right]
 +\int_0^T\sigma\|(u_{xt},\theta_{xt})\|_{L^2}^2dt\leq C,
 \label{normalized-weighted}\\
&\int_0^T\left[\|(v_x,w,w_x,w_t,w_{xx},u_{xx})\|_{L^2}^2+
                          \left\|v-\frac{R\theta}\mu\right\|_{L^2}^2\right]dt\leq C.
\label{normalized-missing-integrals}
\end{align}
For a solution on $[0,\infty)$ satisfying these estimates,
\begin{equation}\label{normalized-qualitative}
\begin{aligned}
\lim_{t\to\infty}\bigg(&\|(w,w_t,v_x,u_t,u_{xx},\theta_x,\theta_t,\theta_{xx})\|_{L^2}
 +\left\|u_x-\frac{R\theta}\mu\right\|_{H^1}\\
 &+\|(v-\bar v,\theta-\bar\theta)\|_{L^\infty}\bigg)=0,
\end{aligned}
\end{equation}
where $\bar v=\int_0^1v\,dx$ and $\bar\theta=\int_0^1\theta\,dx$.
\end{lemma}
\begin{proof}
The differentiated momentum identity retains $v_t=u_x-v$:
\begin{align}
&\frac12\frac d{dt}\|u_t\|_{L^2}^2+\int_0^1\frac\mu v u_{xt}^2dx
 =-\frac d{dt}(e^tP(e^t-1))[u_t]_0^1\notag\\
&\quad-\int_0^1\left[\left(\frac{\alpha\mu u_x}{v\theta}-\frac Rv\right)\theta_t
                      -\frac{\mu u_x-R\theta}{v^2}(u_x-v)\right]u_{xt}dx.
\label{normalized-momentum-time}
\end{align}
The heat identity, with the source fully differentiated, is
\begin{align}
&\frac{c_v}{2}\frac d{dt}\|\theta_t\|_{L^2}^2+\int_0^1\frac\kappa v\theta_{xt}^2dx\notag\\
 &=-\int_0^1\frac\kappa v
           \left(\frac{\beta\theta_t}\theta-\frac{u_x-v}v\right)\theta_x\theta_{xt}dx
 \notag\\
&\quad+\int_0^1\left\{
  \frac{\alpha\mu\theta^{-1}(u_x-R\theta/\mu)^2+R(u_x-R\theta/\mu)}v\right.\notag\\
 &\hspace{25mm}\left.-\frac{R(1-\alpha)}\mu
                   \frac{2\mu(u_x-R\theta/\mu)+R\theta}{v}\right\}\theta_t^2dx
 \notag\\
&\quad+\int_0^1\frac{2\mu(u_x-R\theta/\mu)+R\theta}{v}u_{xt}\theta_tdx\notag\\
 &\quad-\int_0^1\frac{\mu(u_x-R\theta/\mu)^2+R\theta(u_x-R\theta/\mu)}{v^2}
                                            (u_x-v)\theta_tdx.
\label{normalized-thermal-time}
\end{align}
In particular,
\begin{align*}
&\left(u_x-\frac{R\theta}\mu\right)_t
                  =u_{xt}-\frac{R(1-\alpha)}\mu\theta_t,\qquad
 u_x-v=\left(u_x-\frac{R\theta}\mu\right)+\frac{R\theta}\mu-v,\\
&\left|\frac{R\theta}\mu-v\right|+\left|\frac{R\theta}\mu\right|\leq C,\\
&\left\|\left(u_x-\frac{R\theta}\mu\right)^2\right\|_{L^2}^2
 +\left\|\left(u_x-\frac{R\theta}\mu\right)^3\right\|_{L^2}^2\\
&\quad\leq\left[
 \left\|u_x-\frac{R\theta}\mu\right\|_{L^\infty}^2+
 \left\|u_x-\frac{R\theta}\mu\right\|_{L^\infty}^4\right]
                         \left\|u_x-\frac{R\theta}\mu\right\|_{L^2}^2
 \leq C\left\|u_x-\frac{R\theta}\mu\right\|_{H^1}^2.
\end{align*}
The interior term of \eqref{normalized-momentum-time} satisfies
\begin{align*}
&\left|\int_0^1\left[
 \left(\frac{\alpha\mu u_x}{v\theta}-\frac Rv\right)\theta_t
          -\frac{\mu u_x-R\theta}{v^2}(u_x-v)\right]u_{xt}dx\right|\\
&\quad\leq C\left[
 \left(1+\left\|u_x-\frac{R\theta}\mu\right\|_{L^\infty}\right)\|\theta_t\|_{L^2}
 +\left\|\left(u_x-\frac{R\theta}\mu\right)^2\right\|_{L^2}
 +\left\|u_x-\frac{R\theta}\mu\right\|_{L^2}\right]\|u_{xt}\|_{L^2}\\
&\quad\leq\eta\|u_{xt}\|_{L^2}^2+C_\eta\|\theta_t\|_{L^2}^2
 +C_\eta\left\|u_x-\frac{R\theta}\mu\right\|_{H^1}^2
                                      (1+\|\theta_t\|_{L^2}^2).
\end{align*}
For the first and last integrals of \eqref{normalized-thermal-time},
\begin{align*}
&\left|\int_0^1\frac\kappa v
 \left(\frac{\beta\theta_t}\theta-\frac{u_x-v}v\right)
                                      \theta_x\theta_{xt}dx\right|\\
&\quad\leq C\left[\|\theta_x\|_{L^\infty}\|\theta_t\|_{L^2}
 +\left(1+\left\|u_x-\frac{R\theta}\mu\right\|_{L^\infty}\right)
                                      \|\theta_x\|_{L^2}\right]\|\theta_{xt}\|_{L^2}\\
&\quad\leq\eta\|\theta_{xt}\|_{L^2}^2
 +C_\eta\|\theta_x\|_{H^1}^2\|\theta_t\|_{L^2}^2
 +C_\eta\left(\left\|u_x-\frac{R\theta}\mu\right\|_{H^1}^2
                                               +\|\theta_x\|_{L^2}^2\right),\\
&\left|\int_0^1\frac{\mu(u_x-R\theta/\mu)^2+R\theta(u_x-R\theta/\mu)}{v^2}
                                          (u_x-v)\theta_tdx\right|\\
&\quad\leq C\left[
 \left\|\left(u_x-\frac{R\theta}\mu\right)^3\right\|_{L^2}
 +\left\|\left(u_x-\frac{R\theta}\mu\right)^2\right\|_{L^2}
 +\left\|u_x-\frac{R\theta}\mu\right\|_{L^2}\right]\|\theta_t\|_{L^2}\\
&\quad\leq C\left\|u_x-\frac{R\theta}\mu\right\|_{H^1}^2+C\|\theta_t\|_{L^2}^2.
\end{align*}
The remaining two integrals are bounded by
\begin{align*}
&C\int_0^1\left[1+\left|u_x-\frac{R\theta}\mu\right|
                  +\left(u_x-\frac{R\theta}\mu\right)^2\right]\theta_t^2dx\\
&\quad\leq C\|\theta_t\|_{L^2}^2
      +C\left\|u_x-\frac{R\theta}\mu\right\|_{H^1}^2\|\theta_t\|_{L^2}^2,\\
&\left|\int_0^1\frac{2\mu(u_x-R\theta/\mu)+R\theta}{v}u_{xt}\theta_tdx\right|\\
&\quad\leq C\left(1+\left\|u_x-\frac{R\theta}\mu\right\|_{L^\infty}\right)
                                  \|u_{xt}\|_{L^2}\|\theta_t\|_{L^2}\\
&\quad\leq\eta\|u_{xt}\|_{L^2}^2
 +C_\eta\left(1+\left\|u_x-\frac{R\theta}\mu\right\|_{H^1}^2\right)
                                                    \|\theta_t\|_{L^2}^2.
\end{align*}
The boundary term is at most
\[
 \eta\|u_{xt}\|_{L^2}^2+
 C_\eta\left|\frac d{dt}(e^tP(e^t-1))\right|^2.
\]
Choosing $\eta$ below the fixed lower bounds of $\mu/v$ and $\kappa/v$, we obtain
\begin{align}
&\frac d{dt}\left(\frac12\|u_t\|_{L^2}^2+\frac{c_v}{2}\|\theta_t\|_{L^2}^2\right)
                         +c\|(u_{xt},\theta_{xt})\|_{L^2}^2\notag\\
&\leq C\left[\left|\frac d{dt}(e^tP(e^t-1))\right|^2+
          \left\|u_x-\frac{R\theta}\mu\right\|_{H^1}^2
                   +\|\theta_x\|_{L^2}^2+\|\theta_t\|_{L^2}^2\right]\notag\\
&\quad+C\left[\left\|u_x-\frac{R\theta}\mu\right\|_{H^1}^2+
                  \|\theta_x\|_{H^1}^2\right]
                         (\|u_t\|_{L^2}^2+\|\theta_t\|_{L^2}^2).
\label{normalized-coupled-time}
\end{align}
Multiply by $\sigma$ and integrate from $s_j$ with
$s_j\|(u_t,\theta_t)(s_j)\|_{L^2}^2\to0$.
Using \eqref{normalized-base}, \eqref{pressure-weights} and
$\sigma'=\mathbf1_{(0,1)}$, and then letting $j\to\infty$, we obtain
\begin{align*}
&\sup_{0<t\leq T}\sigma(t)\|(u_t,\theta_t)(t)\|_{L^2}^2
       +\int_0^T\sigma\|(u_{xt},\theta_{xt})\|_{L^2}^2dt\\
&\quad\leq C\left[1+\int_0^{\min\{1,T\}}
                              \|(u_t,\theta_t)\|_{L^2}^2dt\right]\\
&\qquad\times\exp\left\{C\int_0^T\left[
 \left\|u_x-\frac{R\theta}\mu\right\|_{H^1}^2+
                                  \|\theta_x\|_{H^1}^2\right]dt\right\}\leq C.
\end{align*}
The spatial estimates follow from
\begin{align*}
\left(u_x-\frac{R\theta}\mu\right)_x
 &=\frac v\mu u_t-
          \left(\frac{\alpha\theta_x}\theta-\frac{v_x}v\right)
                          \left(u_x-\frac{R\theta}\mu\right),\\
\left\|\left(u_x-\frac{R\theta}\mu\right)_x\right\|_{L^2}
 &\leq C\left(\|u_t\|_{L^2}+\left\|u_x-\frac{R\theta}\mu\right\|_{L^2}\right),
\\
u_{xx}&=\left(u_x-\frac{R\theta}\mu\right)_x+
                              \frac{R(1-\alpha)}\mu\theta_x.
\end{align*}
This proves \eqref{normalized-weighted}.

For the remaining integrals,
\begin{align}
&\left(v-\frac{R\theta}\mu\right)_t+
        v-\frac{R\theta}\mu
 =u_x-\frac{R\theta}\mu-\frac{R(1-\alpha)}\mu\theta_t,
 \qquad w_t+w=u_t,\notag\\
&\frac d{dt}\left\|v-\frac{R\theta}\mu\right\|_{L^2}^2+
                   \left\|v-\frac{R\theta}\mu\right\|_{L^2}^2
 \leq C\left(\left\|u_x-\frac{R\theta}\mu\right\|_{L^2}^2+\|\theta_t\|_{L^2}^2\right),
 \notag\\
&\frac d{dt}\|w\|_{L^2}^2+\|w\|_{L^2}^2\leq\|u_t\|_{L^2}^2,
 \qquad w_x=\left(u_x-\frac{R\theta}\mu\right)-
                               \left(v-\frac{R\theta}\mu\right).
\label{normalized-relaxation}
\end{align}
Integrating \eqref{normalized-relaxation} and using
\eqref{normalized-base}, we obtain
\begin{align*}
 &\int_0^T\left[\|(w,w_x)\|_{L^2}^2+
               \left\|v-\frac{R\theta}\mu\right\|_{L^2}^2\right]dt
 \\
 &\quad\leq C+C\int_0^T\left[
 \left\|u_x-\frac{R\theta}\mu\right\|_{L^2}^2+
                          \|(u_t,\theta_t)\|_{L^2}^2\right]dt\leq C.
\end{align*}
Testing \eqref{normalized-effective-gradient} again, with the now fixed
temperature interval, gives
\begin{align*}
&\frac d{dt}\left\|\frac{v_x}v-\frac w\mu\right\|_{L^2}^2+
                  c\left\|\frac{v_x}v-\frac w\mu\right\|_{L^2}^2\\
&\quad\leq C(\|w\|_{L^2}^2+\|\theta_x\|_{L^2}^2)
 +C\alpha^2\int_0^1\left[w^2\theta_t^2+
                                  (1+w_x^2)\theta_x^2\right]dx,\\
&\alpha^2\int_0^T\int_0^1
           [w^2\theta_t^2+(1+w_x^2)\theta_x^2]dxdt\leq C\alpha^2,
 \qquad\int_0^T\|v_x\|_{L^2}^2dt\leq C.
\end{align*}
Finally $w_{xx}=u_{xx}-v_x$ and $w_t=u_t-w$ complete
\eqref{normalized-missing-integrals}.

For the limit, \eqref{normalized-coupled-time},
\eqref{normalized-base} and \eqref{normalized-weighted} imply
\begin{align*}
&\int_1^\infty(\|u_t\|_{L^2}^2+\|\theta_t\|_{L^2}^2)dt\leq C,\\
&\int_1^\infty\left[\frac d{dt}
  \left(\tfrac12\|u_t\|_{L^2}^2+\tfrac{c_v}{2}\|\theta_t\|_{L^2}^2\right)\right]_+dt\\
&\quad\leq C\int_1^\infty\left[
 \left|\frac d{dt}(e^tP(e^t-1))\right|^2+
 \left\|u_x-\frac{R\theta}\mu\right\|_{H^1}^2+
                       \|\theta_x\|_{L^2}^2+\|\theta_t\|_{L^2}^2\right]dt\\
&\qquad+C\sup_{t\geq1}\|(u_t,\theta_t)\|_{L^2}^2
 \int_1^\infty\left[\left\|u_x-\frac{R\theta}\mu\right\|_{H^1}^2+
                                            \|\theta_x\|_{H^1}^2\right]dt
 \leq C.
\end{align*}
Consequently, $\|(u_t,\theta_t)(t)\|_{L^2}\to0$.
Since $\theta_x=0$ at $x=0,1$,
\begin{align*}
\int_1^\infty\left|\frac d{dt}\|\theta_x\|_{L^2}^2\right|dt
&\leq2\int_1^\infty\|\theta_{xx}\|_{L^2}\|\theta_t\|_{L^2}dt\\
&\leq\int_1^\infty(\|\theta_{xx}\|_{L^2}^2+\|\theta_t\|_{L^2}^2)dt\leq C.
\end{align*}
Together with $\int_1^\infty\|\theta_x\|_{L^2}^2dt\leq C$, this gives
$\|\theta_x(t)\|_{L^2}\to0$.
For a scalar function $\mathcal F$ satisfying
$\int_0^\infty|\mathcal F(s)|^pds\leq C$, $1\leq p<\infty$,
H\"older's inequality gives, when $p>1$,
\begin{align*}
\left|\int_0^te^{-c(t-s)}\mathcal F(s)ds\right|
&\leq e^{-ct/2}(t/2)^{1-1/p}
                     \left(\int_0^{t/2}|\mathcal F(s)|^pds\right)^{1/p}\\
&\quad+\left(\frac{p-1}{cp}\right)^{1-1/p}
                     \left(\int_{t/2}^t|\mathcal F(s)|^pds\right)^{1/p}
 \longrightarrow0.
\end{align*}
For $p=1$,
\[
 \left|\int_0^te^{-c(t-s)}\mathcal F(s)ds\right|
 \leq e^{-ct/2}\int_0^{t/2}|\mathcal F(s)|ds+
                         \int_{t/2}^t|\mathcal F(s)|ds\longrightarrow0.
\]
Applying this to \eqref{normalized-relaxation} and to the preceding
estimate for $v_x/v-w/\mu$, we obtain
\begin{equation*}
\lim_{t\to\infty}\left[\|w\|_{L^2}+
 \left\|v-\frac{R\theta}\mu\right\|_{L^2}+\|v_x\|_{L^2}\right]=0.
\end{equation*}
Also, $e^tP(e^t-1)\to0$ by \eqref{pressure-weights}, and
\begin{align*}
&\left\|\frac\mu v\left(u_x-\frac{R\theta}\mu\right)\right\|_{L^\infty}
 \leq e^t|P(e^t-1)|+\|u_t\|_{L^2}\longrightarrow0,\\
&\left\|\left(u_x-\frac{R\theta}\mu\right)_x\right\|_{L^2}
 \leq C\|u_t\|_{L^2}+
 C(\|v_x\|_{L^2}+\|\theta_x\|_{L^2})
                      \left\|u_x-\frac{R\theta}\mu\right\|_{L^\infty}
 \longrightarrow0,\\
&\|u_{xx}\|_{L^2}\leq\left\|\left(u_x-\frac{R\theta}\mu\right)_x\right\|_{L^2}
                         +C\|\theta_x\|_{L^2}\longrightarrow0,\\
&\|\theta_{xx}\|_{L^2}^2\leq C\left(\|\theta_t\|_{L^2}^2+
              \left\|u_x-\frac{R\theta}\mu\right\|_{H^1}^2
                           +\|\theta_x\|_{L^2}^2\right)\longrightarrow0.
\end{align*}
Together with $w_t=u_t-w$ and Poincar\'e's inequality, these give
\eqref{normalized-qualitative}.

\end{proof}
\begin{lemma}\label{lm310}\label{profile-comparison}
The following comparison holds for every finite $T$:
\begin{align}
&\left\|v-A\right\|_{H^1}^2+
 \left\|\theta-\left(\frac{\tilde\mu A}R\right)^{1/(1-\alpha)}\right\|_{H^1}^2
 +\left\|u-\int_0^1u_0dx-A(x-\tfrac12)\right\|_{H^1}^2\notag\\
&\quad\leq C\left[\left\|v-\frac{R\theta}\mu\right\|_{L^2}^2+
        \|v_x\|_{L^2}^2+\|\theta_x\|_{L^2}^2+\|w\|_{L^2}^2+
        \left\|u_x-\frac{R\theta}\mu\right\|_{L^2}^2+e^{2t}|P(e^t-1)|^2\right],
\label{profile-square}\\
&\int_0^T\left[\|v-A\|_{H^1}^2+
  \left\|\theta-\left(\frac{\tilde\mu A}R\right)^{1/(1-\alpha)}\right\|_{H^1}^2
 +\left\|u-\int_0^1u_0dx-A(x-\tfrac12)\right\|_{H^1}^2\right]dt\leq C.
\label{profile-square-integral}
\end{align}
For a solution on $[0,\infty)$ satisfying Lemma~\ref{lm39},
\begin{equation*}
\lim_{t\to\infty}\left(\|(w_x,w_{xx})\|_{L^2}+
 \|v-A\|_{L^\infty}+
 \left\|\theta-\left(\frac{\tilde\mu A}R\right)^{1/(1-\alpha)}\right\|_{L^\infty}\right)=0.
\end{equation*}
\end{lemma}
\begin{proof}
In the remainder of this section $A(t)$ denotes the root in logarithmic time; hence
its physical-time value is $A(\log(1+s))$.
Integration of \eqref{normalized-system} gives
\begin{align}
&c_v\int_0^1\theta dx+\frac12\left\|u-\int_0^1u_0dx\right\|_{L^2}^2
                           +e^tP(e^t-1)\int_0^1v dx\notag\\
&\quad=\int_0^1\left[P(0)v_0+c_v\theta_0+\frac12u_0^2\right]dx
 -\frac12\left(\int_0^1u_0dx\right)^2\notag\\
 &\qquad+\int_0^te^{2s}P'(e^s-1)\int_0^1v(x,s)dx\,ds.\label{normalized-root-energy}
\end{align}
The left side is positive, since $P\geq0$ and $\theta>0$.
For $0\leq\alpha\leq1/2$,
\begin{align*}
&\frac d{da}\left[c_v\left(\frac{\tilde\mu a}R\right)^{1/(1-\alpha)}
                              +\frac{a^2}{24}\right]
 =\frac{c_v}{1-\alpha}\left(\frac{\tilde\mu}R\right)^{1/(1-\alpha)}
                a^{\alpha/(1-\alpha)}+\frac a{12}>0,\qquad a>0,\\
&c_v\left(\frac{\tilde\mu a}R\right)^{1/(1-\alpha)}+\frac{a^2}{24}
 \longrightarrow0\quad(a\downarrow0),\qquad
 \longrightarrow\infty\quad(a\to\infty).
\end{align*}
There is therefore a unique positive root, on the whole finite interval,
\begin{align}
&c_v\left(\frac{\tilde\mu A(t)}R\right)^{1/(1-\alpha)}+\frac{A(t)^2}{24}\notag\\
 &\quad=c_v\int_0^1\theta dx+
       \frac12\left\|u-\int_0^1u_0dx\right\|_{L^2}^2
                         +e^tP(e^t-1)\int_0^1v dx,\label{normalized-root}\\
&0<c\leq A(t)\leq C,\qquad
A'(t)=\frac{e^{2t}P'(e^t-1)\int_0^1v dx}
 {\displaystyle\frac{c_v}{1-\alpha}
       (\tilde\mu/R)^{1/(1-\alpha)}A(t)^{\alpha/(1-\alpha)}+A(t)/12}.
\label{root-bounds}
\end{align}
The constants follow from Lemma~\ref{lm31} and
\eqref{pressure-weights}; in particular $|A'|\leq Ce^{2t}|P'(e^t-1)|$.

The map
$a\mapsto(\tilde\mu a/R)^{1/(1-\alpha)}$ has a bounded derivative
on the fixed positive intervals in \eqref{normalized-base} and
\eqref{root-bounds}. Hence
\begin{align*}
&\left\|\theta-\left(\frac{\tilde\mu\bar v}R\right)^{1/(1-\alpha)}\right\|_{L^2}
 \leq C\left\|\frac{R\theta}\mu-\bar v\right\|_{L^2}
 \leq C\left(\left\|v-\frac{R\theta}\mu\right\|_{L^2}+\|v_x\|_{L^2}\right),\\
&\left\|u-\int_0^1u_0dx-\bar v(x-\tfrac12)\right\|_{L^2}
 \leq\|w\|_{L^2}+C\|v-\bar v\|_{L^2}\leq\|w\|_{L^2}+C\|v_x\|_{L^2},\\
&\left|\frac12\left\|u-\int_0^1u_0dx\right\|_{L^2}^2-
                         \frac{\bar v^2}{24}\right|
 \leq C(\|w\|_{L^2}+\|v_x\|_{L^2}),\\
&\left|c_v\left(\frac{\tilde\mu\bar v}R\right)^{1/(1-\alpha)}+
 \frac{\bar v^2}{24}-c_v\left(\frac{\tilde\mu A}R\right)^{1/(1-\alpha)}-
 \frac{A^2}{24}\right|\\
&\qquad\leq C\left(\left\|v-\frac{R\theta}\mu\right\|_{L^2}+
                         \|v_x\|_{L^2}+\|w\|_{L^2}+e^t|P(e^t-1)|\right),\\
&|\bar v-A|\leq C\left(\left\|v-\frac{R\theta}\mu\right\|_{L^2}+
                        \|v_x\|_{L^2}+\|w\|_{L^2}+e^t|P(e^t-1)|\right).
\end{align*}
The last inequality uses the derivative in \eqref{root-bounds}, bounded
below by $\min\{\bar v,A\}/12$. Finally,
\begin{align*}
&\left\|u-\int_0^1u_0dx-A(x-\tfrac12)\right\|_{L^2}
                         \leq\|w\|_{L^2}+C\|v-A\|_{L^2},\\
&\|u_x-A\|_{L^2}\leq\left\|u_x-\frac{R\theta}\mu\right\|_{L^2}+
 C\left\|\theta-\left(\frac{\tilde\mu A}R\right)^{1/(1-\alpha)}\right\|_{L^2}.
\end{align*}
Squaring proves \eqref{profile-square}.
Integration and \eqref{normalized-missing-integrals} prove
\eqref{profile-square-integral}.
Moreover,
\begin{align*}
&\|w_x\|_{L^2}\leq\left\|u_x-\frac{R\theta}\mu\right\|_{L^2}
 +\left\|v-\frac{R\theta}\mu\right\|_{L^2},\qquad
 \|w_{xx}\|_{L^2}\leq\|u_{xx}\|_{L^2}+\|v_x\|_{L^2}.
\end{align*}
The limit follows from \eqref{profile-square}, \eqref{normalized-qualitative},
and \eqref{normalized-relaxation}.
\end{proof}

\begin{lemma}\label{lm311}
There exist constants $C,\lambda>0$, independent of $T$, such that
\begin{align}
&\left\|v-A(t)\right\|_{H^1}+
 \left\|\theta-\left(\frac{\tilde\mu A(t)}R\right)^{1/(1-\alpha)}\right\|_{H^1}
 +\left\|u-\int_0^1u_0dx-A(t)(x-\tfrac12)\right\|_{H^1}
                        \leq Ce^{-\lambda t},\label{normalized-profile-rate}\\
&\int_0^T\left\|u_x-\frac{R\theta}\mu\right\|_{L^1}dt\leq C.
\label{normalized-residual-integral}
\end{align}
For $1\leq t\leq T$, after increasing $C$ if necessary,
\begin{equation}\label{normalized-full-rate}
\|w\|_{H^1}+\left\|u_x-\frac{R\theta}\mu\right\|_{H^1}
 +\|(u_t,\theta_t,u_{xx},w_{xx},\theta_{xx})\|_{L^2}\leq Ce^{-\lambda t}.
\end{equation}
\end{lemma}
\begin{proof}
{\it Step 1.} Fix a time $t_0\geq1$ and retain $A(t_0)$ as a constant reference.
There exists $\eta>0$, depending only on the fixed compact intervals,
such that, on any interval starting at $t_0$ where
\begin{align}
&\|v-A(t_0)\|_{L^\infty}+
 \left\|\theta-\left(\frac{\tilde\mu A(t_0)}R\right)^{1/(1-\alpha)}\right\|_{L^\infty}
 +\|v_x\|_{L^2}+\|\theta_x\|_{L^2}\notag\\
&\qquad+\left\|u_x-\frac{R\theta}\mu+
                   \frac{e^tP(e^t-1)v}\mu\right\|_{L^2}
                         +e^t|P(e^t-1)|\leq\eta,\label{local-neighborhood}
\end{align}
one has
\begin{align}
&\frac d{dt}\left[\frac12\left\|u_x-\frac{R\theta}\mu+
              \frac{e^tP(e^t-1)v}\mu\right\|_{L^2}^2\right.\notag\\
 &\hspace{15mm}\left.+\frac{(1-\alpha)\tilde\kappa}{2\tilde\mu}
       \left(\frac{\tilde\mu A(t_0)}R\right)^{(\beta-\alpha-1)/(1-\alpha)}
                                       \|\theta_x\|_{L^2}^2\right]\notag\\
&\quad+c\left[\left\|\left(u_x-\frac{R\theta}\mu+
                   \frac{e^tP(e^t-1)v}\mu\right)_x\right\|_{L^2}^2
                                      +\|\theta_t\|_{L^2}^2\right]\notag\\
 &\quad\leq C\left[e^{2t}|P(e^t-1)|^2+
                \left|\frac d{dt}(e^tP(e^t-1))\right|^2\right],
\label{local-dissipation}\\
&\|\theta_{xx}\|_{L^2}\leq C\left[\left\|\left(u_x-\frac{R\theta}\mu+
               \frac{e^tP(e^t-1)v}\mu\right)_x\right\|_{L^2}+
                              \|\theta_t\|_{L^2}+e^t|P(e^t-1)|\right].
\label{local-elliptic}
\end{align}

By \eqref{corrected-residual-boundary},
\begin{align*}
&\left(u_x-\frac{R\theta}\mu+\frac{e^tP(e^t-1)v}\mu\right)_t
 =\left[\frac\mu v\left(u_x-\frac{R\theta}\mu+
                        \frac{e^tP(e^t-1)v}\mu\right)\right]_{xx}\\
 &\hspace{35mm}-\frac{R(1-\alpha)}\mu\theta_t
                   +\left(\frac{e^tP(e^t-1)v}\mu\right)_t,\\
&c_v\theta_t=\frac{\mu u_x}{v}\left(u_x-\frac{R\theta}\mu+
                \frac{e^tP(e^t-1)v}\mu\right)-e^tP(e^t-1)u_x\\
 &\hspace{15mm}+\frac\kappa v\theta_{xx}+
                                  \left(\frac\kappa v\right)_x\theta_x.
\end{align*}
Test the first equation by $u_x-R\theta/\mu+e^tP(e^t-1)v/\mu$ and the second by
\[
\frac{(1-\alpha)R}{\tilde\mu^2}
\left(\frac{\tilde\mu A(t_0)}R\right)^{-2\alpha/(1-\alpha)}\theta_t.
\]
After integration by parts,
\begin{align}
&\frac d{dt}\left[\frac12\left\|u_x-\frac{R\theta}\mu+
                \frac{e^tP(e^t-1)v}\mu\right\|_{L^2}^2\right.\notag\\
 &\hspace{15mm}\left.+\frac{(1-\alpha)\tilde\kappa}{2\tilde\mu}
       \left(\frac{\tilde\mu A(t_0)}R\right)^{(\beta-\alpha-1)/(1-\alpha)}
                                       \|\theta_x\|_{L^2}^2\right]\notag\\
&\quad+\int_0^1\frac\mu v
             \left|\left(u_x-\frac{R\theta}\mu+
                         \frac{e^tP(e^t-1)v}\mu\right)_x\right|^2dx\notag\\
 &\quad+\frac{(1-\alpha)Rc_v}{\tilde\mu^2}
       \left(\frac{\tilde\mu A(t_0)}R\right)^{-2\alpha/(1-\alpha)}
                                       \|\theta_t\|_{L^2}^2\notag\\
&=\int_0^1\left[\frac{(1-\alpha)R}{\tilde\mu^2}
       \left(\frac{\tilde\mu A(t_0)}R\right)^{-2\alpha/(1-\alpha)}
                  \frac{\mu u_x}v-\frac{R(1-\alpha)}\mu\right]\notag\\
 &\hspace{15mm}\times
       \left(u_x-\frac{R\theta}\mu+\frac{e^tP(e^t-1)v}\mu\right)\theta_tdx
 \notag\\
&\quad-\int_0^1\left(\frac\mu v\right)_x
       \left(u_x-\frac{R\theta}\mu+\frac{e^tP(e^t-1)v}\mu\right)
       \left(u_x-\frac{R\theta}\mu+\frac{e^tP(e^t-1)v}\mu\right)_xdx
 \notag\\
&\quad+\int_0^1\left(u_x-\frac{R\theta}\mu+
                 \frac{e^tP(e^t-1)v}\mu\right)
                  \left(\frac{e^tP(e^t-1)v}\mu\right)_tdx
 \notag\\
&\quad-\frac{(1-\alpha)R}{\tilde\mu^2}
       \left(\frac{\tilde\mu A(t_0)}R\right)^{-2\alpha/(1-\alpha)}
                         e^tP(e^t-1)\int_0^1u_x\theta_tdx
 \notag\\
 &\quad+\frac{(1-\alpha)R}{\tilde\mu^2}
       \left(\frac{\tilde\mu A(t_0)}R\right)^{-2\alpha/(1-\alpha)}\notag\\
 &\qquad\times\int_0^1\left\{\left[\frac\kappa v-
    \frac{\tilde\kappa}{A(t_0)}
       \left(\frac{\tilde\mu A(t_0)}R\right)^{\beta/(1-\alpha)}\right]\theta_{xx}
                       +\left(\frac\kappa v\right)_x\theta_x\right\}\theta_tdx
 \notag\\
 &=\sum_{j=1}^5I_j.\label{frozen-exact-energy}
\end{align}
Notice that
\begin{align*}
&\frac{(1-\alpha)R}{\tilde\mu^2}
 \left(\frac{\tilde\mu A(t_0)}R\right)^{-2\alpha/(1-\alpha)}
 \frac R{A(t_0)}\left(\frac{\tilde\mu A(t_0)}R\right)^{1/(1-\alpha)}
 =\frac{R(1-\alpha)}{\tilde\mu}
       \left(\frac{\tilde\mu A(t_0)}R\right)^{-\alpha/(1-\alpha)}.
\end{align*}
The endpoint conditions, \eqref{local-neighborhood} and the compact
coefficient bounds give
\begin{align*}
&\left\|u_x-\frac{R\theta}\mu+\frac{e^tP(e^t-1)v}\mu\right\|_{L^\infty}
 \leq\left\|\left(u_x-\frac{R\theta}\mu+
                    \frac{e^tP(e^t-1)v}\mu\right)_x\right\|_{L^2},\\
 &\|\theta_x\|_{L^\infty}\leq\|\theta_{xx}\|_{L^2},\\
&\left\|\left(\frac\mu v\right)_x\right\|_{L^2}+
  \left\|\left(\frac\kappa v\right)_x\right\|_{L^2}\leq C\eta,
 \qquad \|u_x\|_{L^2}+\|v_t\|_{L^2}\leq C,\\
&\left\|\frac\kappa v-
 \frac{\tilde\kappa}{A(t_0)}
 \left(\frac{\tilde\mu A(t_0)}R\right)^{\beta/(1-\alpha)}\right\|_{L^\infty}\\
&\quad\leq C\left[\|v-A(t_0)\|_{L^\infty}+
 \left\|\theta-
 \left(\frac{\tilde\mu A(t_0)}R\right)^{1/(1-\alpha)}\right\|_{L^\infty}\right]
 \leq C\eta,\\
&\left\|\frac{\mu u_x}v
 \left(u_x-\frac{R\theta}\mu+\frac{e^tP(e^t-1)v}\mu\right)\right\|_{L^2}\\
&\quad\leq C\left[1+
 \left\|u_x-\frac{R\theta}\mu+\frac{e^tP(e^t-1)v}\mu\right\|_{L^\infty}
                       +|e^tP(e^t-1)|\right]
\\
&\qquad\times\left\|u_x-\frac{R\theta}\mu+\frac{e^tP(e^t-1)v}\mu\right\|_{L^2}\\
&\quad\leq C(1+\eta)
 \left\|\left(u_x-\frac{R\theta}\mu+
                  \frac{e^tP(e^t-1)v}\mu\right)_x\right\|_{L^2},\\
&\|\theta_{xx}\|_{L^2}\leq C\left[\|\theta_t\|_{L^2}+
       \left\|\left(u_x-\frac{R\theta}\mu+
                      \frac{e^tP(e^t-1)v}\mu\right)_x\right\|_{L^2}+
                           e^t|P(e^t-1)|\right]\\
&\hspace{25mm}+C\eta\|\theta_{xx}\|_{L^2}.
\end{align*}
Choose $C\eta\leq1/2$ to obtain \eqref{local-elliptic}. The five terms in
\eqref{frozen-exact-energy} are estimated as follows. First,
\begin{align*}
&\frac{(1-\alpha)R}{\tilde\mu^2}
 \left(\frac{\tilde\mu A(t_0)}R\right)^{-2\alpha/(1-\alpha)}
                       \frac{\mu u_x}v-\frac{R(1-\alpha)}\mu\\
&\quad=\frac{(1-\alpha)R}{\tilde\mu^2}
 \left(\frac{\tilde\mu A(t_0)}R\right)^{-2\alpha/(1-\alpha)}
 \left[\frac\mu v\left(u_x-\frac{R\theta}\mu+
                      \frac{e^tP(e^t-1)v}\mu\right)-e^tP(e^t-1)\right]\\
&\qquad+\frac{(1-\alpha)R^2}{\tilde\mu^2}
 \left(\frac{\tilde\mu A(t_0)}R\right)^{-2\alpha/(1-\alpha)}
                              \frac\theta v-\frac{R(1-\alpha)}\mu,\\
&\left|\frac{(1-\alpha)R^2}{\tilde\mu^2}
 \left(\frac{\tilde\mu A(t_0)}R\right)^{-2\alpha/(1-\alpha)}
                            \frac\theta v-\frac{R(1-\alpha)}\mu\right|\\
&\quad\leq C\left[|v-A(t_0)|+
 \left|\theta-\left(\frac{\tilde\mu A(t_0)}R\right)^{1/(1-\alpha)}\right|\right]
 \leq C\eta,\\
|I_1|&\leq C\left[\eta+
 \left\|u_x-\frac{R\theta}\mu+\frac{e^tP(e^t-1)v}\mu\right\|_{L^\infty}\right]
 \left\|u_x-\frac{R\theta}\mu+\frac{e^tP(e^t-1)v}\mu\right\|_{L^2}
                                                   \|\theta_t\|_{L^2}\\
&\leq C\eta
 \left\|\left(u_x-\frac{R\theta}\mu+
                     \frac{e^tP(e^t-1)v}\mu\right)_x\right\|_{L^2}\|\theta_t\|_{L^2}\\
&\leq C\eta\left[
 \left\|\left(u_x-\frac{R\theta}\mu+
               \frac{e^tP(e^t-1)v}\mu\right)_x\right\|_{L^2}^2+\|\theta_t\|_{L^2}^2\right],\\
|I_2|&\leq\left\|\left(\frac\mu v\right)_x\right\|_{L^2}
 \left\|u_x-\frac{R\theta}\mu+\frac{e^tP(e^t-1)v}\mu\right\|_{L^\infty}
 \left\|\left(u_x-\frac{R\theta}\mu+\frac{e^tP(e^t-1)v}\mu\right)_x\right\|_{L^2}\\
&\leq C\eta\left\|\left(u_x-\frac{R\theta}\mu+
               \frac{e^tP(e^t-1)v}\mu\right)_x\right\|_{L^2}^2.
\end{align*}
Fix $c>0$ with
\[
 c\leq\frac\mu v,\qquad
 c\leq\frac{(1-\alpha)Rc_v}{\tilde\mu^2}
       \left(\frac{\tilde\mu A(t_0)}R\right)^{-2\alpha/(1-\alpha)}.
\]
Using the pressure derivative expanded after
\eqref{normalized-stress-heat}, we have
\begin{align*}
|I_3|&\leq C\left|\frac d{dt}(e^tP(e^t-1))\right|
       \left\|\left(u_x-\frac{R\theta}\mu+\frac{e^tP(e^t-1)v}\mu\right)_x\right\|_{L^2}\\
&\quad+C\eta\left\|\left(u_x-\frac{R\theta}\mu+
             \frac{e^tP(e^t-1)v}\mu\right)_x\right\|_{L^2}\|\theta_t\|_{L^2}\\
 &\quad+Ce^t|P(e^t-1)|\left\|\left(u_x-\frac{R\theta}\mu+
             \frac{e^tP(e^t-1)v}\mu\right)_x\right\|_{L^2}\\
&\leq\frac c8\left\|\left(u_x-\frac{R\theta}\mu+
                \frac{e^tP(e^t-1)v}\mu\right)_x\right\|_{L^2}^2\\
&\quad+C\eta\left[\left\|\left(u_x-\frac{R\theta}\mu+
                \frac{e^tP(e^t-1)v}\mu\right)_x\right\|_{L^2}^2+\|\theta_t\|_{L^2}^2\right]\\
&\quad+C\left[e^{2t}|P(e^t-1)|^2+
                       \left|\frac d{dt}(e^tP(e^t-1))\right|^2\right],\\
|I_4|&\leq Ce^t|P(e^t-1)|\|\theta_t\|_{L^2}
 \leq\frac c8\|\theta_t\|_{L^2}^2+Ce^{2t}|P(e^t-1)|^2,\\
|I_5|&\leq C\left[
 \left\|\frac\kappa v-
 \frac{\tilde\kappa}{A(t_0)}
 \left(\frac{\tilde\mu A(t_0)}R\right)^{\beta/(1-\alpha)}\right\|_{L^\infty}
                                      \|\theta_{xx}\|_{L^2}\right.\\
&\hspace{18mm}\left.+
 \left\|\left(\frac\kappa v\right)_x\right\|_{L^2}\|\theta_x\|_{L^\infty}
                                           \right]\|\theta_t\|_{L^2}\\
&\leq C\eta\|\theta_{xx}\|_{L^2}\|\theta_t\|_{L^2}
 \\
 &\leq C\eta\left[\left\|\left(u_x-\frac{R\theta}\mu+
             \frac{e^tP(e^t-1)v}\mu\right)_x\right\|_{L^2}^2+
                   \|\theta_t\|_{L^2}^2+e^{2t}|P(e^t-1)|^2\right].
\end{align*}
Consequently, taking $\eta$ smaller so that $C\eta\leq c/4$,
\begin{align*}
\sum_{j=1}^5|I_j|
&\leq\frac c2\left[
 \left\|\left(u_x-\frac{R\theta}\mu+
                  \frac{e^tP(e^t-1)v}\mu\right)_x\right\|_{L^2}^2+\|\theta_t\|_{L^2}^2\right]\\
&\quad+C\left[e^{2t}|P(e^t-1)|^2+
                  \left|\frac d{dt}(e^tP(e^t-1))\right|^2\right].
\end{align*}
Substitution into \eqref{frozen-exact-energy} proves
\eqref{local-dissipation}. The boundary conditions and
\eqref{local-elliptic} also imply
\begin{align}
&\left\|u_x-\frac{R\theta}\mu+\frac{e^tP(e^t-1)v}\mu\right\|_{L^2}^2
                                +\|\theta_x\|_{L^2}^2\notag\\
&\quad\leq C\left[\left\|\left(u_x-\frac{R\theta}\mu+
             \frac{e^tP(e^t-1)v}\mu\right)_x\right\|_{L^2}^2+
                    \|\theta_t\|_{L^2}^2+e^{2t}|P(e^t-1)|^2\right].
\label{local-spectral}
\end{align}

{\it Step 2.} 
Choose $0<\delta<\varepsilon/2$. By \eqref{pressure-weights} there is
a fixed $C_P$ such that, for every $t_0\geq1$,
\begin{align}
&\sup_{s\geq t_0}e^{2s}|P(e^s-1)|^2
 +\int_{t_0}^\infty e^{\delta(s-t_0)}
 \left[e^{2s}|P'(e^s-1)|+e^{2s}|P(e^s-1)|^2\right.\notag\\
 &\hspace{48mm}\left.+
              \left|\frac d{ds}(e^sP(e^s-1))\right|^2\right]ds
 \leq C_Pe^{-\delta t_0},\label{pressure-tail}\\
&|A(t)-A(t_0)|+
 \left|\left(\frac{\tilde\mu A(t)}R\right)^{1/(1-\alpha)}-
        \left(\frac{\tilde\mu A(t_0)}R\right)^{1/(1-\alpha)}\right|
 \leq C C_Pe^{-\delta t_0},\qquad t\geq t_0.\label{root-drift}
\end{align}
Keep $A(t_0)$ fixed in \eqref{local-dissipation}.
Combining \eqref{local-dissipation} with \eqref{local-spectral} and
choosing $0<\omega<\min\{\delta,1/2\}$ sufficiently small yields
\begin{align}
&e^{\omega(t-t_0)}\left[\left\|u_x-\frac{R\theta}\mu+
                      \frac{e^tP(e^t-1)v}\mu\right\|_{L^2}^2+\|\theta_x\|_{L^2}^2\right]
 \notag\\
&\quad+\int_{t_0}^te^{\omega(s-t_0)}
 \left[\left\|\left(u_x-\frac{R\theta}\mu+
            \frac{e^sP(e^s-1)v}\mu\right)_x\right\|_{L^2}^2
                       +\|\theta_t\|_{L^2}^2+\|\theta_{xx}\|_{L^2}^2+\|\theta_x\|_{L^2}^2\right]ds
 \notag\\
&\quad\leq C\left[\left\|u_x(t_0)-\frac{R\theta(t_0)}{\mu(t_0)}+
                 \frac{e^{t_0}P(e^{t_0}-1)v(t_0)}{\mu(t_0)}\right\|_{L^2}^2
                           +\|\theta_x(t_0)\|_{L^2}^2+C_Pe^{-\delta t_0}\right].
\label{local-weighted-rate}
\end{align}
Here $C$ is independent of $t_0$ for $c\leq A(t_0)\leq C$ and $0\leq\alpha\leq1/2$.
The momentum identity and fixed coefficient bounds give
\begin{align}
&\left\|\left(u_x-\frac{R\theta}\mu\right)_x\right\|_{L^2}^2+\|u_t\|_{L^2}^2\notag\\
&\quad\leq C\left[\left\|\left(u_x-\frac{R\theta}\mu+
          \frac{e^tP(e^t-1)v}\mu\right)_x\right\|_{L^2}^2+e^{2t}|P(e^t-1)|^2\right],
 \notag\\
&\left\|u_x-\frac{R\theta}\mu\right\|_{L^2}^2
 \leq C\left[\left\|u_x-\frac{R\theta}\mu+
                  \frac{e^tP(e^t-1)v}\mu\right\|_{L^2}^2+e^{2t}|P(e^t-1)|^2\right].
\label{local-residual-conversions}
\end{align}
Apply the weighted energy inequality for $z_t+z=b$,
\begin{equation}\label{damping-weighted}
e^{\omega(t-t_0)}\|z(t)\|_{L^2}^2+
 (1-\omega)\int_{t_0}^te^{\omega(s-t_0)}\|z(s)\|_{L^2}^2ds
 \leq\|z(t_0)\|_{L^2}^2+
                \int_{t_0}^te^{\omega(s-t_0)}\|b(s)\|_{L^2}^2ds,
\end{equation}
to the three equations
\begin{align}
&\left(v-\frac{R\theta}\mu\right)_t+v-\frac{R\theta}\mu
 =u_x-\frac{R\theta}\mu-\frac{R(1-\alpha)}\mu\theta_t,\notag\\
&(v_x)_t+v_x=\left(u_x-\frac{R\theta}\mu\right)_x+
                            \frac{R(1-\alpha)}\mu\theta_x,
 \qquad w_t+w=u_t.\label{three-damping-equations}
\end{align}
Together with \eqref{profile-square} and
\eqref{local-weighted-rate}--\eqref{local-residual-conversions}, this gives
\begin{align}
&\|v-A(t)\|_{H^1}^2+
 \left\|\theta-\left(\frac{\tilde\mu A(t)}R\right)^{1/(1-\alpha)}\right\|_{H^1}^2
 +\left\|u-\int_0^1u_0dx-A(t)(x-\tfrac12)\right\|_{H^1}^2\notag\\
&\leq Ce^{-\omega(t-t_0)}\left[
 \|v(t_0)-A(t_0)\|_{H^1}^2+
 \left\|\theta(t_0)-\left(\frac{\tilde\mu A(t_0)}R\right)^{1/(1-\alpha)}\right\|_{H^1}^2
 \right.\notag\\
&\hspace{35mm}\left.+\left\|u(t_0)-\int_0^1u_0dx-A(t_0)(x-\tfrac12)\right\|_{H^1}^2
                      +C_Pe^{-\delta t_0}\right].\label{provisional-profile-rate}
\end{align}
By \eqref{centered-w} and the bounded derivative of
$\theta\mapsto R\theta/\mu$,
\begin{align*}
&\|w(t_0)\|_{L^2}^2+\|v_x(t_0)\|_{L^2}^2
       +\left\|v(t_0)-\frac{R\theta(t_0)}{\mu(t_0)}\right\|_{L^2}^2\\
&\quad+\left\|u_x(t_0)-\frac{R\theta(t_0)}{\mu(t_0)}+
                   \frac{e^{t_0}P(e^{t_0}-1)v(t_0)}{\mu(t_0)}\right\|_{L^2}^2\\
&\leq C\left[\|v(t_0)-A(t_0)\|_{H^1}^2+
 \left\|\theta(t_0)-
       \left(\frac{\tilde\mu A(t_0)}R\right)^{1/(1-\alpha)}\right\|_{H^1}^2
 \right.\\
&\hspace{12mm}\left.+
 \left\|u(t_0)-\int_0^1u_0dx-A(t_0)(x-\tfrac12)\right\|_{H^1}^2
                         +e^{2t_0}|P(e^{t_0}-1)|^2\right].
\end{align*}
Sobolev embedding and \eqref{root-drift} consequently imply
\begin{align}
&\left[\|v-A(t_0)\|_{L^\infty}+
 \left\|\theta-\left(\frac{\tilde\mu A(t_0)}R\right)^{1/(1-\alpha)}\right\|_{L^\infty}
                 +\|v_x\|_{L^2}+\|\theta_x\|_{L^2}\right.\notag\\
&\hspace{15mm}\left.+\left\|u_x-\frac{R\theta}\mu+
       \frac{e^tP(e^t-1)v}\mu\right\|_{L^2}+e^t|P(e^t-1)|\right]^2\notag\\
&\leq C\left[\|v(t_0)-A(t_0)\|_{H^1}^2+
 \left\|\theta(t_0)-\left(\frac{\tilde\mu A(t_0)}R\right)^{1/(1-\alpha)}\right\|_{H^1}^2
 \right.\notag\\
&\hspace{20mm}\left.+\left\|u(t_0)-\int_0^1u_0dx-A(t_0)(x-\tfrac12)\right\|_{H^1}^2
                          +C_Pe^{-\delta t_0}\right].\label{neighborhood-improvement}
\end{align}

Let $C_I$ be a fixed bound in \eqref{profile-square-integral}, and
let $C_4$ dominate the constants in \eqref{neighborhood-improvement}.
Choose
\begin{equation}\label{entry-choices}
0<\xi\leq\min\{1,\eta^2/(64C_4)\},\qquad
S_0\geq1,\quad C_Pe^{-\delta S_0}\leq\xi.
\end{equation}
If $T\geq S_0+1+C_I/\xi$, averaging \eqref{profile-square-integral} over
$[S_0,S_0+1+C_I/\xi]$ supplies $t_0\in[S_0,S_0+1+C_I/\xi]$ such that
\begin{align*}
&\|v(t_0)-A(t_0)\|_{H^1}^2+
 \left\|\theta(t_0)-\left(\frac{\tilde\mu A(t_0)}R\right)^{1/(1-\alpha)}\right\|_{H^1}^2\\
&\hspace{25mm}+\left\|u(t_0)-\int_0^1u_0dx-A(t_0)(x-\tfrac12)\right\|_{H^1}^2
 \leq\frac{C_I}{S_0+1+C_I/\xi-S_0}\leq\xi.
\end{align*}
At $t_0$, the left side of \eqref{local-neighborhood} is less than
$\eta/2$. On its maximal interval below $\eta$,
\eqref{neighborhood-improvement} bounds its square by
$2C_4\xi\leq\eta^2/32$. Continuity therefore excludes a first
exit, and \eqref{provisional-profile-rate} holds on $[t_0,T]$.
On $[0,\min\{t_0,T\}]$ the uniform bounds, enlarged by $e^{\omega(S_0+1+C_I/\xi)}$,
give the same rate. This also covers $T<S_0+1+C_I/\xi$, and proves
\eqref{normalized-profile-rate} with $\lambda=\omega/2$.
Moreover,
\begin{align*}
&\int_0^{\min\{T,S_0+1+C_I/\xi\}}\left\|u_x-\frac{R\theta}\mu\right\|_{L^1}dt\\
&\quad\leq\sqrt{S_0+1+C_I/\xi}\left(\int_0^T
                 \left\|u_x-\frac{R\theta}\mu\right\|_{L^2}^2dt\right)^{1/2}\leq C,\\
&\int_{S_0+1+C_I/\xi}^T\left\|u_x-\frac{R\theta}\mu\right\|_{L^1}dt\\
&\quad\leq C\int_{S_0+1+C_I/\xi}^Te^{-\lambda t}dt\leq C
                           \qquad(T\geq S_0+1+C_I/\xi).
\end{align*}
This proves \eqref{normalized-residual-integral}.

{\it Step 3.} Equations
\eqref{local-weighted-rate}--\eqref{local-residual-conversions} give
\begin{align}
\int_{t_0}^Te^{\omega(s-t_0)}\left[
\|u_t\|_{L^2}^2+\|\theta_t\|_{L^2}^2+
 \left\|u_x-\frac{R\theta}\mu\right\|_{H^1}^2+
 \|\theta_x\|_{H^1}^2+
 \left|\frac d{ds}(e^sP(e^s-1))\right|^2\right]ds\leq C.
\label{late-weighted-derivatives}
\end{align}
For $t\geq t_0+1$, choose $s\in[t-1,t]$ with
$\|u_t(s)\|_{L^2}^2+c_v\|\theta_t(s)\|_{L^2}^2\leq Ce^{-\omega(t-t_0)}$.
The integrating factor in \eqref{normalized-coupled-time} then gives
\begin{align*}
&\|u_t(t)\|_{L^2}^2+\|\theta_t(t)\|_{L^2}^2\\
&\quad\leq C\left[e^{-\omega(t-t_0)}+
 \int_{t-1}^t\left(\left|\frac d{ds}(e^sP(e^s-1))\right|^2+
  \left\|u_x-\frac{R\theta}\mu\right\|_{H^1}^2\right.\right.\\
&\hspace{46mm}\left.\left.+
  \|\theta_x\|_{L^2}^2+\|\theta_t\|_{L^2}^2\right)ds\right]\\
&\qquad\times\exp\left\{C\int_{t-1}^t
 \left(\left\|u_x-\frac{R\theta}\mu\right\|_{H^1}^2+
                      \|\theta_x\|_{H^1}^2\right)ds\right\}
 \leq Ce^{-\omega(t-t_0)}.
\end{align*}
The reverse momentum and heat estimates give
\begin{align*}
&\left\|\left(u_x-\frac{R\theta}\mu\right)_x\right\|_{L^2}^2+\|u_{xx}\|_{L^2}^2
 \leq C\left(\|u_t\|_{L^2}^2+
       \left\|u_x-\frac{R\theta}\mu\right\|_{L^2}^2+\|\theta_x\|_{L^2}^2\right),\\
&\|\theta_{xx}\|_{L^2}^2\leq C\left(\|\theta_t\|_{L^2}^2+
       \left\|u_x-\frac{R\theta}\mu\right\|_{H^1}^2+\|\theta_x\|_{L^2}^2\right),
 \qquad w_{xx}=u_{xx}-v_x.
\end{align*}
Together with \eqref{three-damping-equations}, these prove
\eqref{normalized-full-rate} on $[t_0+1,T]$. On
$[1,\min\{t_0+1,T\}]$, use \eqref{normalized-weighted} and increase $C$
by the fixed factor $e^{\lambda(S_0+1+C_I/\xi+1)}$.
\end{proof}

By Lemmas~\ref{lm31}--\ref{lm311}, there are fixed $c,C_5,C_6>0$ such that
\begin{align}
&c\leq v,\theta\leq C_5,\qquad
 \sup_{t\leq T}\|(u,\theta,v)\|_{H^1}^2+
 \int_0^T\|(v_x,w_x,\theta_x,\theta_t,\theta_{xx})\|_{L^2}^2dt\notag\\
 &\hspace{35mm}+\int_0^T\sigma\|\theta_{xt}\|_{L^2}^2dt\leq C_5,\notag\\
&\int_0^T\left\|u_x-\frac{R\theta}\mu\right\|_{L^1}dt\leq C_6.
\label{normalized-total-estimate}
\end{align}
The second estimate follows from Lemma~\ref{lm311}. Choose
\begin{align}
M&\geq4\max\{2,C_5+C_6,c^{-1},M_0^2,
                       \underline v^{-1},\underline\theta^{-1}\},\notag\\
0<\varepsilon_0&\leq
       \min\{\tfrac12,\eta_0(2M)^{-5/2},(2M)^{-\beta-2}\}.
\label{normalized-threshold}
\end{align}
For $0\leq\alpha\leq\varepsilon_0$, \eqref{normalized-total-estimate} gives
\begin{align}
&\frac2M<v,\theta<\frac M2,\notag\\
&\sup_{t\leq T}\|(u,\theta,v)\|_{H^1}^2+
 \int_0^T\left\|w_x+v-\frac{R\theta}\mu\right\|_{L^1}dt
 \leq C_5+C_6\leq\frac M4,\notag\\
&\int_0^T\|(v_x,w_x,\theta_x,\theta_t,\theta_{xx})\|_{L^2}^2dt
       +\int_0^T\sigma\|\theta_{xt}\|_{L^2}^2dt
 \leq C_5\leq\frac M4.
\label{normalized-strict-improvement}
\end{align}
Thus the bounds defining $Y(0,T;M)$ improve strictly to those defining
$Y(0,T;M/2)$. Continuity and Proposition~\ref{local-property} exclude a
finite first exit and give a single solution for all normalized times.

Restoring the checks in \eqref{time-change}, the physical derivatives are
\begin{align}
&v_x(x,s)=e^t\check v_x(x,t),\quad v_s(x,s)=\check u_x(x,t),\quad
 v_{xs}(x,s)=\check u_{xx}(x,t),\notag\\
&u_x=\check u_x,\quad u_{xx}=\check u_{xx},\quad
 u_s=e^{-t}\check u_t,\notag\\
 &\theta_x=\check\theta_x,\quad\theta_{xx}=\check\theta_{xx},\quad
 \theta_s=e^{-t}\check\theta_t,\qquad ds=e^t dt.
\label{physical-derivatives}
\end{align}
For every finite $S>0$, put $T=\log(1+S)$. First,
\[
 \int_0^T\|\check u_x\|_{L^2}^2dt
 \leq2T\sup_{t\leq T}\|\check v\|_{L^2}^2+
                         2\int_0^T\|\check w_x\|_{L^2}^2dt\leq C(1+T).
\]
By \eqref{physical-derivatives},
\begin{align*}
\int_0^S\|v_x\|_{L^2}^2ds
 &=\int_0^Te^{3t}\|\check v_x\|_{L^2}^2dt
 \leq Ce^{3T}=C(1+S)^3,\\
\int_0^S\|(u_s,\theta_s)\|_{L^2}^2ds
 &=\int_0^Te^{-t}\|(\check u_t,\check\theta_t)\|_{L^2}^2dt\leq C,\\
\int_0^S\|(u_x,\theta_x,u_{xx},\theta_{xx},v_{xs})\|_{L^2}^2ds
 &=\int_0^Te^t\|(\check u_x,\check\theta_x,
               \check u_{xx},\check\theta_{xx},\check u_{xx})\|_{L^2}^2dt\\
 &\leq Ce^T(1+T)
 =C(1+S)[1+\log(1+S)],\\
\int_0^S\|v_s\|_{H^1}^2ds
 &=\int_0^Te^t\|\check u_x\|_{H^1}^2dt
 \leq C(1+S)[1+\log(1+S)].
\end{align*}
These are the finite-time integral estimates in \eqref{th2-vu}.
Continuity follows from \eqref{time-change} and
Proposition~\ref{local-property}.
The unique root of \eqref{normalized-root} is exactly the pullback of
the root of \eqref{th2-A}, by \eqref{normalized-root-energy}. Therefore
\begin{align*}
&\left\|\begin{pmatrix}
\displaystyle\frac{v(x,s)}{1+s}-A(\log(1+s))\\
\displaystyle u(x,s)-\int_0^1u_0dx-A(\log(1+s))(x-\tfrac12)\\
\displaystyle\theta(x,s)-
\left(\frac{\tilde\mu A(\log(1+s))}R\right)^{1/(1-\alpha)}
\end{pmatrix}\right\|_{H^1}\\
&\qquad\leq Ce^{-\lambda\log(1+s)}=C(1+s)^{-\lambda}.
\end{align*}
Writing the root in physical time again gives \eqref{th2-contr}.
This completes the proof of Theorem~\ref{tm12}.

\end{document}